\documentclass[10pt,a4paper,twoside]{amsart}
\usepackage[T1]{fontenc}
\usepackage[utf8]{inputenc}
\usepackage[english]{babel}
\usepackage{amsmath}
\usepackage{amsthm}
\usepackage{amssymb}
\usepackage{amsfonts}
\usepackage{amsxtra}
\usepackage{enumerate}
\usepackage{mathrsfs}
\usepackage{verbatim}
\usepackage{color}
\usepackage{ulem}

\newcommand{\N}{\mathbb N}
\newcommand{\Z}{\mathbb Z}
\newcommand{\Q}{\mathbb Q}
\newcommand{\R}{\mathbb R}

\newtheorem*{lem}{Lemma}

\newtheorem*{props}{Proposition}

\newtheorem{theorem}{Theorem}[section]
\newtheorem{lemma}[theorem]{Lemma}
\newtheorem{cor}[theorem]{Corollary}

\newtheorem{prop}[theorem]{Proposition}

\theoremstyle{remark}
\newtheorem{remark}[theorem]{Remark}

\theoremstyle{definition}
\newtheorem{definition}[theorem]{Definition}

\DeclareMathOperator*{\supp}{supp}

\DeclareMathOperator*{\dist}{dist}
\DeclareMathOperator*{\dom}{dom}

\DeclareMathOperator*{\inte}{int}

\DeclareMathOperator*{\id}{id}
\DeclareMathOperator*{\ev}{ev}

\DeclareMathOperator*{\EV}{ev}

\let\phi=\varphi
\let\e=\varepsilon

\begin{document}

\title[Dual solutions for Lorentzian cost functions]{On the existence of dual solutions for Lorentzian cost functions}
\author{Martin Kell}
\address{Martin Kell,\newline\indent Fachbereich Mathematik, Universit\"at T\"ubingen, 
\newline\indent Auf der Morgenstelle 10, 72076 T\"ubingen, Germany}
\email{martin.kell@math.uni-tuebingen.de}
\author{Stefan Suhr}
\address{Stefan Suhr\newline\indent Ruhr-Universit\"at Bochum, Fakult\"at f\"ur Mathematik\newline\indent
Geb\"aude NA 4/33, D-44801 Bochum, Germany}
\email{stefan.suhr@ruhr-uni-bochum.de}

\begin{abstract}
The dual problem of optimal transportation in Lorentz-Finsler geometry is studied. It is shown that in general no solution exists even in the presence
of an optimal coupling. Under natural assumptions dual solutions are established. It is further shown that the existence of a dual solution implies that the 
optimal transport is timelike on a set of full measure. In the second part the persistence of absolute continuity along an optimal transportation under obvious 
assumptions is proven and a solution to the relativistic Monge problem is provided.
\end{abstract}
\date{\today}
\maketitle

\section{Introduction}

The theory of optimal transportation on Riemannian manifolds has revolutionized Riemannian geometry during the last decade 
with its characterization of lower bounds on the Ricci curvature in terms of optimal transport and the formulation of synthetic Ricci curvature for 
metric measure spaces. Einstein's field equations, the central equations of general relativity, are equations for the Ricci curvature 
of a Lorentzian metric. Thus the prospect of developing generalized notions of spacetimes and solutions to the Einstein field equations
readily motivates a theory of optimal transportation in Lorentzian geometry. The works \cite{EcksteinMiller,mccann18,Miller,suhr16} are 
first steps in this direction, with the work by McCann giving a first characterization of lower Ricci curvatue bounds for globally hyperbolic spacetimes. 
Optimal transportation in the context of special relativity was proposed in \cite{brenier} and studied in \cite{berpue,BPP,LPZ}. 

The present theory is formulated for globally hyperbolic Lorentz-Finsler spacetimes. See Section \ref{results} and \cite{besu17} for definitions and properties. 
The cost function in Lorentz-Finsler geometry is the negative of the time separation, or Lorentzian distance function, for future causally related points and
extended by $\infty$ to non-future causally related points. Because of the distance-like character of this cost function the problem is a relativistic
version of the original Monge problem. This ensures that optimal couplings of finite cost transport along future pointing causal geodesics. 
The non-finiteness, non-Lipschitzity and the discontinuity of the cost function at the boundary of $J^+$ cause additional difficulties, though. 

This article is a continuation of the work \cite{suhr16} by the second author on optimal transportation in Lorentz-Finsler geometry. 
The first major result, Theorem \ref{thm1}, focuses on the existence of solutions to the dual problem of Lorentzian optimal transportation, also known as the 
{\it dual Kantorovich problem}, see \cite{villani}. In \cite{suhr16} the second author gave natural conditions on the marginals to obtain 
optimal couplings and the weak Kantorovich duality. Here the weak Kantorovich duality says that the ``$\inf$-$\sup$-equality'' holds.
No statement on the existence of solutions was made, though. The dual problem does in general not admit a solution as Section 
\ref{1-dim-counterexample} shows. The problem lies on the lightcones. In general only a negligible part of mass is allowed to be transported along lightlike 
geodesics if a solution is to exist. To circumvent the underlying phenomenon the condition of {\it strict timelikeness} is introduced in 
Definition \ref{defstricttime}. Theorem \ref{thm1} then shows that dual solutions exists if the marginals satisfy the strict timelikeness condition
and at least one marginal has connected support. The conditions are met on a weakly dense subset of pairs of measures by Corollary \ref{cor1}.
Theorem \ref{thm1} generalizes results in \cite{berpue,BPP,mccann18}. The condition of strict timelikeness is related to parts of the definition
of {\it $q$-separated measures} in \cite{mccann18}. Conversely the existence of a dual solution necessitates that only a negligible part of mass 
is transported along lightlike geodesics, see Theorem \ref{thm2} the second major result of the present paper. A related result in special 
relativity is \cite[Theorem C]{BPP}.

It should be noted that both the existence and non-existence of dual solutions adapt to the Lorentzian cost with $q\in(0,1)$
as studied in \cite{EcksteinMiller,mccann18}. Indeed, the proof of Theorem \ref{thm1} makes only use of the causality structure.
In Example \ref{1-dim-counterexample} and Theorem \ref{thm2} the adjusted upper bound is $-\frac{C}{2} n^{1-\frac{q}{2}}$ and one may 
replace Lemma \ref{DKP-interpolation} by McCann's result on starshapedness of $q$-separated measures \cite[Proposition 5.5]{mccann18}.

An important question for Lorentzian optimal transportation is whether the interpolation measures of an optimal
transport are absolutely continuous with respect to the volume measure if at least one marginal is absolutely continuous. 
Theorem \ref{thm4}, the third major result, shows that intermediate measures are absolutely continuous if one marginal is 
absolutely continuous and the other marginal is concentrated on an achronal set. This result seems optimal also in the 
non-relativistic setting, as the non-uniqueness of optimal couplings and interpolation measures usually prevent intermediate 
measures to be absolutely continuous. However, the optimal couplings constructed from the solution of the relativistic Monge
problem can be used to show that there are indeed absolutely continuous intermediate measure though they are non-unique
even assuming that transport is along time-affinely parametrized geodesics.  We emphasize that the proof of Theorem \ref{thm4}
does not rely on the Lipschitz regularity of the transport directions as e.g. in \cite{bebu2}, since Lipschitz regularity is not available, see \cite{suhr16}.

We remark that the synthetic proof of existence of optimal transport maps adapts easily to Lorentzian cost
functions with $q\in (0,1)$. Indeed, excluding lightlike geodesics one
may parametrized geodesics by arclength. Then the non-branching property (Lemma \ref{lem:strong non-branching}) and
a weak measure contraction property (Lemma \ref{L1}) or alternatively the $(K,N)$-convexity of the entropy as 
obtained in \cite{mccann18} are for example sufficient to follow mutatis mutandis the proof of Cavalletti-Huesmann \cite{CH2015NBTrans}.

The last major result, Theorem \ref{thm5}, provides a solution to the relativistic Monge problem. It is shown that there exists an optimal transport map between any two causally related measures whenever the first measure is 
absolutely continuous with respect to any volume form of the differentiable manifold. If the second measure is 
concentrated on an achronal set this was already proven by the second author, see \cite[Theorem 2.12]{suhr16} 
which also contains a uniqueness statement. Note, however, the existence proof in this article is independent 
of \cite{suhr16} and only relies on a non-branching property of time-affinely parametrized geodesics, 
see Lemma \ref{lem:strong non-branching} below. Uniqueness then follows using \cite[Proposition 3.21]{suhr16} 
which can be an seen as a stronger non-branching property that is related to the volume form.

The article is organized as follows: In Section \ref{results} the setting is introduced and the main results are formulated. In Section \ref{examples}
two examples are given. One example shows that not all pairs of measures with an optimal coupling admit a solution to the dual problem. 
The second example shows that the dual solution does not need to be Lipschitz, i.e. the optimal transport is not bounded away from the lightcones
even though there exists a strictly timelike coupling. Finally Section \ref{proofs} contains the proofs of all results.

{\it Acknowledgements:} The second author is partially supported by the SFB/TRR 191: Symplectic Structures in Geometry, Algebra and Dynamics.

\section{results}\label{results}

Let $M$ be a smooth manifold. Throughout the article one fixes a complete Riemannian metric $h$ on $M$, though local changes to the metric will 
be allowed. Consider a continuous function $\mathbb{L}\colon TM\to \R$, smooth on $TM\setminus T^0M$ (here $T^0M$ denotes the zero section in $TM$) and
positive homogenous of degree $2$ such that the second fiber derivative is non-degenerate with index $\dim M-1$. Let $\mathcal{C}$ be a causal structure of 
$(M,\mathbb{L})$, see \cite{suhr16}, and define the Lagrangian $L$ on $TM$ by setting 
$$L(v):=\begin{cases}
-\sqrt{\mathbb{L}(v)},& \text{ for }v\in \mathcal{C},\\
\infty,&\text{ otherwise.}
\end{cases}$$
The pair $(M,L)$ is referred to as a {\it Lorentz-Finsler manifold}. 

One calls an absolutely continuous curve $\gamma\colon I\to M$ {\it ($\mathcal{C}$-)causal} if $\dot\gamma\in \mathcal{C}$ for almost all $t\in I$.
Note that this condition already  implies that the tangent vector is contained in $\mathcal{C}$ whenever it exists. 

Denote with $J^+(x)$ the set of points $y\in M$ such that there exists a causal curve $\gamma\colon [a,b]\to M$ with $\gamma(a)=x$ and $\gamma(b)=y$. 
Two points $x$ and $y$ will be called {\it causally related} if $y \in J^+(x)$. Note that this relation is in general asymmetric. Define the set
$$J^+:=\{(x,y)\in M\times M|\; y\in J^+(x)\}$$  
and $J^-(y):\{x\in M|\; y\in J^+(x)\}$.

A Lorentz-Finsler manifold is said to be {\it causal} if it does not admit a closed causal curve, i.e. $J^+\cap \triangle =\emptyset$ for $\triangle:=\{(x,x)|\; x\in M\}$.

\begin{definition}
A causal Lorentz-Finsler manifold $(M,L)$ is {\it globally hyperbolic} if the sets $J^+(x)\cap J^-(y)$ are compact for all $x,y\in M$. 
\end{definition}

By \cite{besu17} every globally hyperbolic Lorentz-Finsler spacetime admits a diffeomorphism (called a {\it splitting}) $M\cong \R\times N$ such that the 
projection $\tau\colon \R\times N\to \R$, $(\theta,p)\mapsto \theta$ satisfies 
$$-d\tau(v) \le \min \{L(v),-|v|\}$$ 
for all $v\in \mathcal{C}$. In the following one fixes a splitting $\tau$ and refers to it as a {\it time function}. Note that though the proofs below use a particular highly 
non-unique time function, the existence and uniqueness results do not depend on the choice of such a function.  

Define the {\it Lagrangian action (relative to $L$)} of a absolutely continuous curve $\gamma\colon [a,b]\to M$:
$$\mathcal{A}(\gamma):=\int_a^b L(\dot\gamma)dt \in \R\cup\{\infty\}$$
Note that $\mathcal{A}(\gamma)\in \R$ if and only if $\gamma$ is causal. The following result is proven in the same fashion as in the Lorentzian case, see \cite{suhr16}.

\begin{prop}\label{P1}
Let $(M,L)$ be globally hyperbolic. Then for every pair $(x,y)\in J^+$ there exists a minimizer of $\mathcal{A}$ with finite 
action connecting the two points. The minimizer $\gamma$ solves the Euler-Lagrange equation of $\mathbb{L}$ up to monotone reparametrization and one 
has $\dot\gamma\in\mathcal{C}$ everywhere.
\end{prop}

The Euler-Lagrange equation of $\mathbb{L}$ defines a maximal local flow 
$$\Phi^\mathbb{L}\colon \mathbb{U}\to TM,$$
where $\mathbb{U}\subset \R\times TM$ is an open neighborhood of $\{0\}\times (TM\setminus T^0M)$. Note that $\mathcal{C}$ and $\partial \mathcal{C}$ are 
invariant under $\Phi^\mathbb{L}$. A curve $\gamma\colon I\to M$ is a {\it $\Phi^\mathbb{L}$-orbit} if it solves the Euler-Lagrange equation of $\mathbb{L}$, 
see \cite{suhr16}.

For a globally hyperbolic Lorentz-Finsler manifold $(M,L)$ define the {\it Lorentzian cost function} 
\begin{align*}
c_L\colon M\times M&\to \R\cup\{\infty\}\\
x,y&\mapsto\min\left\{\left.\mathcal{A}(\gamma)\right|\; \gamma\text{ connects $x$ and $y$}\right\}.
\end{align*}
It is immediate that $c_L$ is non-positive for causally related points and infinite otherwise. 

Define $\overline{\tau}\colon M\times M\to \R$, $(x,y)\mapsto \tau(y)-\tau(x)$. Let $\mathcal{P}_\tau(M\times M)$ be the space of of Borel probability measures 
$\pi$ on $M\times M$ such that $\overline{\tau}\in L^1(\pi)$. The {\it Lorentzian cost} is the functional 
$$\mathcal{P}_\tau(M\times M)\to \R\cup\{\infty\},\; \pi \mapsto \int c_Ld\pi.$$

The minimization problem for the Lorentzian cost is called the {\it Relativistic Monge-Kantorovich problem}:
Given two Borel probability measures $\mu_0$ and $\mu_1$ on $M$ find a minimizer of the Lorentzian cost 
among all Borel probability measures on $M\times M$ with first marginal $\mu_0$ and second marginal $\mu_1$. 
Any minimizer will be called an {\it optimal coupling} between $\mu_0$ and $\mu_1$.

Let $\mathcal{P}(M)$ denote the space of Borel probability measures on $M$. For a splitting $\tau\colon M\to \R$ define
$$\mathcal{P}_\tau^+(M):=\{(\mu_0,\mu_1)|\; \tau\in L^1(\mu_0)\cap L^1(\mu_1)\text{ and $\mu_0$, $\mu_1$ are 
$J^+$-related}\}
$$
where two probability measures are {\it $J^+$-related} (or just {\it causally related}) if there exists a 
coupling $\pi$ with $\pi(J^+)=1$. Note that if $\pi$ is a coupling of two $J^+$-related probability measures $\mu_0$ and $\mu_1$, then 
$(\mu_0,\mu_1)\in \mathcal{P}^+_\tau(M)$ if and only if $\pi \in \mathcal{P}_\tau(M\times M)$.

A function $\psi\colon M\to \R\cup\{\infty\}$ with $\psi\not\equiv \infty$ is {\it $c_L$-convex} if there exists a function $\zeta \colon M\to \R\cup\{\pm\infty\}$ such that
$$\psi(x)=\sup\left\{\zeta(y)-c_L(x,y)|\; y\in M\right\}$$ 
for all $x\in M$. The function 
\begin{align*}
\psi^{c_L}\colon M&\to \R\cup \{-\infty\}\\
y&\mapsto \inf \left\{ \psi(x)+c_L(x,y)|\; x\in M\right\}
\end{align*}
is called the {\it $c_L$-transform of $\psi$}. A pair $(x,y)\in M\times M$ belongs to the {\it $c_L$-subdifferential} $\partial_{c_L} \psi$ if 
$$\psi^{c_L}(y)-\psi (x)=c_L(x,y).$$

\begin{prop}[\cite{suhr16}]\label{P2a}
Let $(\mu,\nu)\in \mathcal{P}^+_\tau(M)$.  One has
$$\inf\left\{\left.\int c_{L}d\pi\right|\pi\text{ is a coupling of $\mu$ and }\nu\right\}= \sup\left\{\int_{M} \phi(y) d\nu(y)-\int_M \psi(x)d\mu(x)\right\}$$
where the supremum is taken over the functions $\psi\in L^1(\mu),\phi\in L^1(\nu)$ with $\phi(y)-\psi(x)\le c_L(x,y)$.
\end{prop}
Proposition \ref{P2a} shows 
that the {\it weak Kantorovich duality} holds.

\begin{definition}
Let $(\mu,\nu)\in \mathcal{}P^+_\tau(M)$. 
A $c_L$-convex function 
$$\psi\colon M\to \R\cup \{\infty\}$$ 
is a {\it solution to the dual Kantorovich 
problem (DKP) for $(\mu,\nu)$} if $\psi$ is $\mu$-almost surely finite  
and 
$$\psi^{c_L}(y)-\psi(x)=c_L(x,y)$$
$\pi$-almost surely for every optimal coupling $\pi$ of $\mu$ and $\nu$. 
\end{definition}

\begin{definition}[\cite{villani}]
A {\it dynamical coupling} of two probability measures $\mu_0$ and $\mu_1$ is a probability measure $\Pi$ on the space of continuous curves $\eta\colon [0,1]\to M$
such that $(\ev_{0})_\sharp\Pi=\mu_{0}$ and $(\ev_{1})_\sharp\Pi=\mu_{1}$. 
\end{definition}
Dynamical couplings in Lorentzian geometry have been studied in \cite{EcksteinMiller,Miller}.

\begin{definition}\label{defstricttime}
A pair $(\mu,\nu)$ of probability measures is {\it strictly timelike} if there exists a dynamical coupling $\Pi$ supported in the subspace of causal curves such that 
$(\partial_t\ev)_\sharp (\Pi\times \mathcal{L}|_{[0,1]})$ is locally bounded away from $\partial\mathcal{C}$ where $\partial_t\ev(\gamma,t):= \dot\gamma(t)$. 
\end{definition}

\begin{remark}
\begin{itemize}
\item[(1)] Recall that every causal curve admits a Lipschitz parameterization. Further the condition of strict timelikeness is convex in the sense that the set of strictly timelike 
pairs of measures is convex. 
\item[(2)] The condition of {\it strict timelikeness} generalizes the {\it supercritical speed} for {\it relativistic cost functions} in \cite{berpue,BPP,LPZ}.
It is further related to the condition of {\it $q$-separatedness} in \cite{mccann18}.
\end{itemize}
\end{remark}

\begin{theorem}[Existence of dual solutions]\label{thm1}
Let $(\mu,\nu)\in \mathcal{P}^+_\tau(M)$ be strictly timelike and assume that $\supp \mu$ is connected. Then the DKP for $(\mu,\nu)$ has a 
solution. More precisely for every optimal coupling $\pi$ there exists a $c_L$-convex function $\psi\colon M\to\R\cup\{\infty\}$ real-valued on $\supp\mu$ such that 
$\pi$-almost everywhere $\psi^{c_L}(y)-\psi(x)=c_L(x,y)$, i.e. $\supp\pi \subset \partial_{c_L}\psi$. 
\end{theorem}

The theorem generalizes \cite[Theorem 5.13]{berpue} and \cite[Theorem B]{BPP}. Note that Theorem \ref{thm1} is most likely optimal, as Theorem \ref{thm2} 
shows that dual solutions cannot exist whenever there are optimal couplings transporting a set of positive measure along $\partial\mathcal{C}$, the boundary 
of the lightcone.  A $1+1$-dimensional example of the non-existence of dual solution is provided in Section \ref{1-dim-counterexample}
below. 

\begin{cor}\label{cor1}
Every pair $(\mu_0,\mu_1)\in \mathcal{P}_\tau^+(M)$ can be approximated in the weak topology by a sequence $\{(\mu_0^n,\mu_1^n)\}_{n\in\N}\subset
\mathcal{P}_\tau^+(M)$ such that every pair $(\mu_0^n,\mu_1^n)$ admits a solution to the DKP.
\end{cor}

\begin{proof}[Proof of Corollary {\ref{cor1}}]
Choose a vector field $X\in \Gamma(TM)$ with $L(X)<0$ and $d\tau (X)=1$. Denote with $\phi_t^X$ the flow of $X$. Then 
$$(\nu_0^n,\nu_1^n):=((\phi_{-t}^X)_\sharp \mu_0,\mu_1)\in \mathcal{P}_\tau^+(M)$$
is strictly timelike for all $n\in\N$. The measure $\nu_0^n$ can approximated by a measure $o^n$ with connected support and 
$$\supp \nu_0^n\subset \supp o^n.$$
Now 
$$(\mu_0^n,\mu_1^n):=\left(\frac{n-1}{n}\nu_0^n+\frac{1}{n}o^n, \frac{n-1}{n}\nu_1^n+\frac{1}{n}(\phi_1^X)_\sharp o^n\right)$$
satisfies the assumptions of Theorem \ref{thm1}.
\end{proof}
 
If $c_L(x,y)=0$ then by Proposition \ref{P1} there exists a lightlike $\Phi^\mathbb{L}$-orbit $\eta$, i.e. $L(\dot\eta)\equiv0$, which cannot be parametrized by 
``arclength'', i.e. $L(\dot \eta)\equiv 1$. In particular, such lightlike $\Phi^\mathbb{L}$-orbits do not admit a ``preferred affine parametrization'' in any sense. However, 
using the time function $\tau$ one may reparametrize every causal $\Phi^\mathbb{L}$-orbit as follows: Denote with $\Gamma$ 
the set of causal minimizers $\gamma\colon [0,1]\to M$ of $\mathcal{A}$ such that $d\tau(\dot\gamma)\equiv \text{const}(\gamma)$.
Elements of $\Gamma$ are called {\it time-affinely} parametrized geodesics.
For $(x,y)\in J^+$ consider the subspace 
$$\Gamma_{x\to y}:=\{\gamma\in \Gamma|\; \ev\nolimits_0(\gamma)=x,\; \ev\nolimits_1(\gamma)=y\}$$ 
where 
$$\EV\colon \Gamma\times [0,1]\to M,\; (\gamma,t)\mapsto \gamma(t)\text{ and }\ev\nolimits_t:=\ev(.,t).$$
Since $(M,L)$ is assumed to be globally hyperbolic one has $\Gamma_{x\to y} \ne \varnothing$.

\begin{definition}
A Borel measure $\Pi$ on $\Gamma$ is a {\it dynamical optimal coupling}  of $\mu_0:=(\ev_0)_\sharp\Pi$ and $\mu_1:=(\ev_1)_\sharp \Pi$ 
if $\pi:=(\ev_0,\ev_1)_\sharp \Pi$ is an optimal coupling between $\mu_0$ and $\mu_1$. 
\end{definition}

\begin{prop}[\cite{suhr16}]\label{dynoptcou}
Let $(\mu_0,\mu_1)\in \mathcal{P}^+_\tau(M)$. Then there exists a dynamical optimal coupling $\Pi$ 
for $\mu_0$ and $\mu_1$ with $\supp\Pi\subseteq \Gamma$.
\end{prop}

In the following all Lebesgue measures are understood to be induced by the Riemannian metric $h$. The assumptions $(i)$ and
$(ii)$ of the next theorem are similar and yield the same conclusion. It is not necessarily obvious 
that Theorem \ref{thm2}$(i)$ is the analogue of the classical solution to the dual Kantorovich problem for real-valued cost functions, see \cite{villani}. 
The conclusion of Theorem \ref{thm2} under assumption $(ii)$ on the other hand has no counterpart there. 

\begin{theorem}[Non-existence of dual solutions]\label{thm2}
Let $(\mu,\nu)\in \mathcal{P}^+_\tau(M)$. Assume that the supports of both measures are disjoint and that the DKP for $(\mu,\nu)$ admits a solution $\psi\colon 
M\to\R\cup\{\infty\}$. Further assume $\mu$ to be either 
\begin{enumerate}
\item[(i)] supported on a spacelike hypersurface $H$ and that it is absolutely continuous with respect to the 
Lebesgue measure $\mathcal{L}_H$ on $H$ or
\item[(ii)]\label{thm3} absolutely continuous with respect to the Lebesgue measure $\mathcal{L}$ on $M$.
\end{enumerate}
Denote with $\Pi$ a dynamical optimal coupling of $\mu$ and $\nu$ and with $\Gamma_0$ the set of lightlike 
minimizers in $\Gamma$. Then $\Pi(\Gamma_0)=0$, i.e. only a $\mu$-negligible set of points is transported along lightlike minimizers. 
\end{theorem}

The theorem generalizes \cite[Corollary 3.6]{berpue} and \cite[Theorem C]{BPP}. 
Note that Theorem \ref{thm2} is proven indirectly and relies on a very similar construction as the
$1+1$-dimensional example in Section \ref{1-dim-counterexample}.

The following two theorems are the second main result of this article. The first theorem has a counterpart in 
the work of the second author, see \cite[Theorem 2.13]{suhr16} and the second theorem is a solution to
the {\it relativistic Monge problem}: For $(\mu,\nu)\in \mathcal{P}^+_\tau(M)$ find a Borel-mesurable map $F\colon M\to M$ such that $\pi:=(\id,F)_\sharp \mu$ is
an optimal coupling of $\mu$ and $\nu$. 

Note that the proofs are independent of \cite[Theorem 2.13]{suhr16} and rely only on a straightforward geometric argument, see \cite[Proposition 3.22]{suhr16}.

\begin{definition}\label{defachro}
A set $A\subset M$ is {\it achronal} if $c_L|_{A\times A}\ge 0$. In case $c_L|_{A\times A}\equiv \infty$ one says that $A$
is {\it acausal}. 
\end{definition}

It is not difficult to see that any time slice $\{\tau = \tau_0\}$ is acausal. The definition is in accordance with the classical definitions of 
acausal and achronal sets in Lorentzian geometry.

\begin{theorem}[Existence and uniqueness for achronal targets]\label{thm4}
Let $(\mu,\nu)\in \mathcal{P}^+_\tau(M)$ such that $\mu$ is absolutely continuous 
with respect to the Lebesgue measure on $M$ and $\nu$ is concentrated on an achronal set.
Then there exists a unique dynamical optimal coupling $\Pi$ such that $(\ev_t)_\sharp \Pi$ is absolutely 
continuous for $t\in[0,1)$ and the optimal couplings $(\ev_t,\ev_1)_\sharp \Pi$ are induced by transport maps.
\end{theorem}
\begin{remark}
Note that the Monge problem is in general highly non-unique even in the non-relativistic setting. 
In the non-relativistic setting the equivalent to being supported on a time slice would be to assume 
the second measure is concentrated in a level set of a dual solution to the Monge problem. However, 
such a condition depends on the first measure.
\end{remark}

\begin{theorem}[Solution to the relativistic Monge Problem]\label{thm5}
Let $(\mu,\nu)\in \mathcal{P}^+_\tau(M)$ such that $\mu$ is absolutely continuous with respect 
to the Lebesgue measure on $M$. Then there exists a Borel-measurable map $F\colon M\to M$ such that $\pi:=(\id,F)_\sharp \mu$ is an optimal coupling of $\mu$ and 
$\nu$.
\end{theorem}


\section{Two Examples}\label{examples}

\subsection{An example with no dual solution}\label{1-dim-counterexample}

Let $M=\R^2$ and 
$$\mathbb{L}\colon TM\to \R,\; (x,v)\mapsto v_1^2-v_2^2$$ 
with
$$\mathcal{C}:=\{(x,v)\in TM|\; v_2\ge |v_1|\},$$
where $v=(v_1,v_2)$. It follows that $c_L$ is the negative Lorentzian distance on the $2$-dimensional Minkowski space. 
Fix the splitting 
$$\tau\colon M\to \R,\; (s,t)\mapsto t.$$

Denote with $i_0,i_1\colon \R\to M$, the maps $i_0(s):=(s,0)$ and $i_1(s):=(s,1)$, respectively and with 
$\mathcal{L}_1$ the Lebesgue measure on the real line $\R$. Consider the transport problem between 
$$\overline{\mu}:=(i_0)_\sharp(\mathcal{L}_1|_{[0,1]})\text{ and }\overline{\nu}:=(i_1)_\sharp(\mathcal{L}_1|_{[1,2]}).$$
The map 
$$\overline{T}\colon M\to M,\; (s,t)\mapsto (s+1,t+1)$$
induces a causal coupling $(\id,\overline{T})_\sharp \overline{\mu}$ of $\overline{\mu}$ and $\overline{\nu}$, i.e. $(\overline{\mu},\overline{\nu})\in 
\mathcal{P}^+_\tau(M)$.

\begin{prop}\label{propex}
The DKP for $(\overline{\mu}, \overline{\nu})$ does not have a solution.
\end{prop}

The transport problem for $(\overline{\mu},\overline{\nu})$ is equivalent to the following transport problem on the real line: The restriction of $c_L$ to 
$$\{((s,0),(t,1))|\; s,t\in \R\}\subset \R^2\times \R^2$$
and the identification 
$$\{((s,0),(t,1))|\; s,t\in \R\}\cong \R\times \R,\; ((s,0),(t,1))\cong (s,t)$$
yield the cost function 
$$c\colon \R\times\R\to \R,\; 
(s,t)\mapsto \begin{cases} -\sqrt{1-(s-t)^2},&\text{ for }|s-t|\le 1\\ 
\infty,&\text{ for }|s-t|\ge 1\end{cases}
$$ 
and the probability measures $\overline{\mu}$ and $\overline{\nu}$ are identified with 
$$\mu=\mathcal{L}_1|_{[0,1]}\text{ and }\nu=\mathcal{L}_1|_{[1,2]},$$ 
respectively.

\begin{lemma}\label{lemmaex1}
If $\pi$ is a coupling of $\mu$ and $\nu$ with finite $c$-cost, then 
$$\pi= (\id,T)_\sharp \mu,$$ 
where $T\colon \R\to \R$, $s\mapsto s+1$.
\end{lemma}

\begin{proof}
Let $\pi$ be a coupling of $\mu$ and $\nu$ with finite cost. For every $\e>0$ one has 
$$\mu([0,\e])=\nu([1,1+\e])=\nu([-1,1+\e]).$$
The support of $\pi$ is contained in $\{(s,t)|\;|s-t|\le 1\}$ since it has finite cost. Therefore 
$$\mu([0,\e])=\pi([0,\e]\times \R)=\pi([0,\e]\times[1,1+\e]).$$ 
By complementary reasoning one
concludes that 
$$\pi([\e,1]\times \R)=\pi([\e,1]\times[1+\e,2]).$$
An induction over $n$ then implies that the support of $\pi$ is contained in 
$$\bigcup_{k=0}^{2^n}\left([k\cdot 2^{-n},(k+1)\cdot2^{-n}]\times [1+k\cdot2^{-n},1+(k+1)\cdot2^{-n}]\right)$$
for every $n\in \N$. The claim follows in the limit for $n\to\infty$.
\end{proof}

\begin{lemma}\label{lemmadots}
Let $[a,b]\subset \R$, $\e>0$ and a Borel measurable set $B\subset [a,b]$ be given with $\mathcal{L}_1(B)\ge \e(b-a)$. Then for all $n\in\N$ there exists 
$\{t_i\}_{1\le i\le n}\subset B$ with $t_1<\ldots <t_n$ and $t_{i+1}-t_i\ge \frac{\e}{2n}(b-a)$.
\end{lemma}

\begin{proof}
Let $n\in\N$ be given. Consider the function $t\mapsto \mathcal{L}_1(B\cap [a,t])$ and choose $t_1\in B$ such that 
$$\mathcal{L}_1(B\cap [a,t_1])\in\left(0,\frac{\e}{2n}(b-a)\right).$$
Then one has
$$\mathcal{L}_1\left(B\cap \left[a,t_1+\frac{\e}{2n}(b-a)\right)\right)\le \frac{\e}{n}(b-a).$$ 
Next consider the function 
$$t\mapsto \mathcal{L}_1\left(B\cap \left[t_1+\frac{\e}{2n}(b-a),t\right]\right),\text{ for }t>t_1+\frac{\e}{2n}(b-a).$$ 
Choose $t_2\in B$ such that 
$$\mathcal{L}_1\left(B\cap \left[t_1+\frac{\e}{2n}(b-a),t_2\right]\right)\in \left(0,\frac{\e}{2n}(b-a)\right).$$ 
Then one has
$$\mathcal{L}_1\left(B\cap \left[a,t_2+\frac{\e}{2n}(b-a)\right]\right)\le 2\frac{\e}{n}(b-a)=2\frac{\e}{n}(b-a).$$
Continue inductively. For $k<n$ one has 
$$\mathcal{L}_1\left(\left[a,t_k+\frac{\e}{2n}(b-a)\right]\right)\le k\frac{\e}{n}(b-a)=\frac{k}{n}\e(b-a)\le \frac{n-1}{n}\e(b-a).$$
Thus one concludes $t_k+\frac{\e}{2n}(b-a)<b$ for all $k<n$. This shows that the construction does not 
terminate before $n$ points have been chosen. The claimed properties are clear from the construction.
\end{proof}

\begin{lemma}\label{lemmaex2}
There does not exists a $c$-convex function $\psi\colon \R\to \R\cup\{\infty\}$ with $\psi|_{[0,1]}\not \equiv \infty$ and 
$\psi^c(y)-\psi(x)=c(x,y)$ for $\pi$-almost all $(x,y)\in \R\times \R$, where $\pi$ is the coupling in Lemma \ref{lemmaex1}.
\end{lemma}

The existence of a solution to the dual problem is independent of an additive constant in the definition of the cost function, i.e. for $c':=c+1$ the pair 
$(\phi',\psi')$ solves the DKP for $c'$ iff $(\phi',\psi-1)$ solves the DKP for $c$.

\begin{proof}
Let $(\phi,\psi)$ be a solution to the DKP for $(\mu,\nu)$, i.e. $c(x,y)\ge \phi(x)+\psi(y)$ and 
$$\int c\, d\pi' =\int \phi\, d\mu +\int \psi\, d\nu= \int (\phi +\psi)\, d\pi'.$$
Thus one has $c=\phi+\psi$ $\pi'$-almost surely.

Let $s<t\in [0,1]$. By Lemma \ref{lemmadots} for every $n\in\N$ there exist $s\le t_1<\ldots t_n\le t$ with $t_{k+1}-t_k\ge \frac{t-s}{2n}$ and $\phi(t_k)+\psi(t_k+1)=0$.
Thus one has 
$$c\left(t_{k+1},1+t_k\right)\le c\left(0,1-\frac{1}{2n}(t-s)\right)\le -\sqrt{\frac{1}{2n}(t-s)}.$$
As in \cite[page 61]{villani} it follows that
\begin{align*}
\psi(s)&\le \psi(t)+\sum_{k=1}^{n} c\left(t_{k+1},1+t_k\right)\\
&\le \psi(t)-n\sqrt{\frac{1}{2n}(t-s)}=\psi(t)-\sqrt{\frac{n(t-s)}{2}}
\end{align*}
for all $n$. Therefore $\psi(s)=-\infty$ for all $s<1$. But this contradicts the definition of $\psi$.
\end{proof}

\begin{proof}[Proof of Proposition \ref{propex}]
The claim follows directly from Lemma \ref{lemmaex2} by reversing the identification 
$$\{((s,0),(t,1))|\; s,t\in\R\}\cong \R\times\R.$$
\end{proof}


\subsection{An example with non-Lipschitz dual solution}
An example is given of a strictly timelike pair $(\mu,\nu)$ for which the optimal coupling is not bounded away from $\partial J^+$. 
This counters the intuition that the optimal coupling of strictly timelike pairs is supported away from $\partial J^+$.

Let $0<\e<\frac 1 2$. Choose a function $\bar f\in C^{0,\frac{1}{2}}([0,5])\cap C^\infty ([0,5]\setminus \{2\})$ with 
\begin{itemize}
\item[(1)] $\bar f\equiv 1+\e$ on $[0,1]$, 
\item[(2)] $\bar f(x)> -c_L((x,0),(1,1))$ for $x\in [1,2)$, 
\item[(3)] $\bar f(2)=0$
\item[(4)] $\bar f(x)> c_L((3,0),(x,1))$ for $x\in (2,3]$, 
\item[(5)] $\bar{f}'<0$ near $2$.
\item[(6)] $\bar f\equiv \e-1$ on $[3,4]$ 
\end{itemize}
that induces a $\frac{1}{2}$-H\"older continuous function $f$ on $\R/5\Z$, smooth except at $[2]\in \R/5\Z$. 

Now consider $\R/5\Z\times \R$ with the inner product 
$$(\mathbb{L}=)g:=d\theta^2 -dt^2$$ 
for $(\theta,t)\in \R/5\Z\times\R$ where 
$$\mathcal{C}:=\{v|\; g(v,v)\le 0,\; dt(v)\ge 0\}.$$ 
The cost function $c_L$ for the pair $(g,\mathcal{C})$ is 
$$c_L((\eta,s),(\theta,t))=\begin{cases}-\sqrt{(t-s)^2-(\theta-\eta)^2},&\; s\le t,\; t-s\ge \theta-\eta\\
\infty,&\text{ else.}
\end{cases}$$ 
Define 
$$\phi\colon \R/5\Z\times \R \to \R\cup\{\infty\},\; \phi(y):= \inf \{f(\theta)+c_L((\theta,0),y)|\; \theta\in \R/5\Z\}.$$
\begin{lemma}
One has 
$$\phi^{c_L}(y)= \inf\nolimits_x\{\phi(x)+c_L(x,y)\}\equiv \phi(y)$$
for all $y\in \R/5\Z\times \R$. It follows that $\phi$ is $c_L$-concave. 
\end{lemma}

\begin{proof}
Indeed first since $c_L(x,x)=0$ one has 
$$\phi^{c_L}(y)=\inf\nolimits_x \{\phi(x)+c_L(x,y)\}\le \phi(y)$$
for all $y$. 

Fix $y\in \R/5\Z\times \R$ and choose $z\in \R/5\Z\times \R$ with $\phi^{c_L}(y)=\phi(z)+c_L(z,y)$. For $z$ choose $\theta\in \R/5\Z$ 
with $\phi(z)=f(\theta)+c_L((\theta,0),z)$. Then one has 
$$\phi(y)\le f(\theta)+c_L((\theta,0),y)\le f(\theta) +c_L((\theta,0),z)+c_L(z,y)=\phi(z)+c_L(z,y)=\phi^{c_L}(y)$$
by the triangle inequality for $c_L$. Thus one has 
$$\phi(y)=\inf\nolimits_x\{\phi(x)+c_L(x,y)\}=\phi^{c_L}(y)$$ 
for all $y$.
\end{proof}

As usual define 
$$\partial_c \phi :=\{(x,y)|\; \phi(y)-\phi(x)=c_L(x,y)\}\subseteq (\R/5\Z\times\R)\times (\R/5\Z\times\R)$$
and $\partial_c\phi_{x}:=p_2(\partial_c\phi \cap (\{x\}\times (\R/5\Z\times\R)))$. Note that for all $(\theta',t)$ with $t\ge 0$ there exists $\theta \in \R/5\Z$ with 
$(\theta',t)\in \partial_c\phi_{(\theta,0)}$ since $c_L$ is continuous on its domain. 

\begin{lemma}
For $\theta\neq [2]$ and $y\in \partial_c \phi_{(\theta,0)}$ with $t(y)>0$ one has $y\in I^+(\theta,0)$. Further for every $(\theta,t)\in\R/5\Z\times [0,1]$ with $\theta\neq [2]$ the set 
$\partial_c\phi_{(\theta,0)}\cap \R/5\Z\times\{t\}$ has exactly one element. 
\end{lemma} 

\begin{proof}
Indeed $y\in \partial_c \phi_{(\theta,0)}$ implies 
$$\inf\nolimits_{\eta}  \{f(\eta)+c_L((\eta,0),y)\}=\phi(y)=f(\theta)+c_L((\theta,0),y),$$ 
i.e. the function $\eta\mapsto  f(\eta)+c_L((\eta,0),y)$ has a minimum in $\theta$. If $c_L((\theta,0),y)=0$ then $\eta\mapsto c_L((\eta,0),y)$ falls off
to one side of $\theta\neq [2]$ faster than $f$ can rise by construction. Therefore in this case $\theta$ cannot be a minimum. Thus it follows that $c_L((\theta,0),y)<0$,
i.e. $y\in I^+((\theta,0))$. 

Now fix $\theta\neq [2]$ and $t\in [0,1]$. Then the equation 
$$\frac{\partial}{\partial\theta}f(\theta)+ \frac{\partial}{\partial\theta} c_L((\theta,0),(\theta',t))=0$$ 
has exactly one solution 
$\theta'$. Since by the previous paragraph the points in $\partial_c\phi_{(\theta,0)}$ are characterized as solutions to this equation, the second part of the claim follows.
\end{proof}

\begin{lemma}\label{L10}
For every neighborhood $U$ of $\partial J^+$ and every $t\in (0,1]$, the $1$-dimensional Lebesgue measure of 
$$\{\theta\in\R/5\Z\setminus\{[2]\}|\;((\theta,0),y_{(\theta,t)})\in U\}$$ 
is positive, where 
$y_{(\theta,t)}$ denotes the unique point in $\partial_c\phi_{(\theta,0)}\cap \R/5\Z\times\{t\}$.
\end{lemma}

\begin{proof}
The Lebesgue measure of points $\theta$ such that $-f'(\theta)\ge C$ is bounded from below by the Lebesgue measure of the set of points with 
$\frac{\partial}{\partial\theta} c_L((\theta,0),(1,1))\ge C$ for $C$ sufficiently large by the assumptions (2) and (3) above. The last set has positive Lebesgue measure for every $C<\infty$. For every neighborhood $U$
of $\partial J^+$ there exists $C_U<\infty$ such that $\left|\frac{\partial}{\partial\theta} c_L((\theta,0),y)\right|\ge C_U$ for all $((\theta,0),y)\in U$. Now $y=y_{(\theta,t)}$ is the unique solution to 
the equation $\frac{\partial}{\partial\theta}f(\theta)+ \frac{\partial}{\partial\theta} c_L((\theta,0),y)=0$ with $t(y)=t$, and the claim follows.
\end{proof}

\begin{lemma}\label{L11}
There exists $\delta=\delta(t)>0$ such that  $\dist((\theta -t,t),\partial_c \phi_{(\theta,0)})\ge \delta$ 
for all $\theta\neq [2]$.
\end{lemma}

\begin{proof}
For $\theta\neq [2]$ let $(\theta_t,t)\in \partial_c \phi_{(\theta,0)}$. By the Lemma \ref{L10} one has that 
$$\eta\mapsto c_L((\eta,0),(\theta_t,t))$$ 
is smooth at $\theta$ and 
$$\frac{\partial}{\partial\theta}f(\theta)+ \frac{\partial}{\partial\theta} c_L((\theta,0),(\theta_t,t))=0\Leftrightarrow\frac{\partial}{\partial\theta} c_L((\theta,0),(\theta_t,t))
=-\frac{\partial}{\partial\theta}f(\theta).$$ 
Since $\frac{\partial}{\partial\theta}f(\theta)$ is bounded outside every neighborhood of $[2]\in \R/5\Z$, the existence of $\delta$ follows there. By the assumption 
$\frac{\partial}{\partial\theta}f<0$ on a neighborhood of $[2]$ the bound follows in fact for all $\theta\neq [2]$ since $\frac{\partial}{\partial\theta} c_L((\theta,0),(\theta',t)) \to -\infty$ for $\theta'\downarrow \theta -t$. 
\end{proof}

Now consider the probability measure $\mu:= I_\sharp \mathcal{L}'$ where 
$$I\colon \R/5\Z\hookrightarrow \R/5\Z\times\R,\; \theta\mapsto (\theta, 0).$$ 
and $\mathcal{L}'$ is the normalized Lebesgue measure on $\R/5\Z$.
The following result is a reformulation of \cite[Proposition 3]{bebu1} adapted to the present situation. 

\begin{props}[\cite{bebu1}]
Let $\phi$ be a $c_L$-convex function, and let $\mu$ be a probability measure on $M$. Then there exists a probability measure 
$\nu$ on $M$ such that $\phi$ solves the DKP for $(\mu,\nu)$.
\end{props}

By the proposition there exists a probability measure $\nu$ supported on $\R/5\Z\times \{1\}$ such that $\phi$ is optimal for the pair $(\mu,\nu)$. By Lemma \ref{L10} 
the transport is not bounded away from $\partial J^+$.

Now a rotation of $\R/5\Z$ in the negative direction leaves $\mu$ unchanged, but the coupling induced 
by $\phi$ is twisted into a coupling whose support has positive distance from $\partial J^+$ by Lemma \ref{L11}. Thus the pair $(\mu,\nu)$ is strictly timelike.

\section{Proofs}\label{proofs}

\subsection{Proof of Theorem \ref{thm1}}

Let $(\mu,\nu)\in \mathcal{P}^+_\tau(M)$ be strictly timelike. Further let $\pi$ be an optimal coupling of $\mu$ and $\nu$. Note that $\pi$ is causal 
since its cost is finite. Fix $(x_0,y_0)\in \supp \pi$. Define
\begin{align*}
\psi\colon J^-(\supp\nu)&\to\R\cup\{\pm\infty\},\\ 
x&\mapsto\sup\left\{\sum_{i=0}^k [c_L(x'_i,y'_i)-c_L(x'_{i+1},y'_i)]\right\}
\end{align*}
where the supremum is taken over all $k\in \N$ and all sequences  
$$\{(x'_i,y'_i)\}_{0\le i\le k+1}\subset\supp \pi\text{ with }x'_{k+1}=x\text{ and }(x'_0,y'_0)=(x_0,y_0).$$
One has $\psi(x_0)\ge c_L(x_0,y_0)-c_L(x_0,y_0)=0$. At the same time the right hand side of the above definition is nonpositive for $x=x_0$ 
by cyclic monotonicity, see \cite[Proposition 2.7]{suhr16}. Therefore one has $\psi(x_0)=0$. 

Next one shows that $\psi$ is real-valued and measurable on $\supp \mu$. Fix $(u,w)\in \supp\pi$. Consider for $k\in \N$ chains 
$\{(u_i,w_i)\}_{0\le i\le k}\subseteq \supp \pi$ with $u_{i+1}\in J^-(w_i)\cap \supp\mu$ for $0\le i\le k-1$ where $(u_0,w_0)=(u,w)$. Define 
$$A:=\{w_k|\; \{(u_i,w_i)\}_{0\le i\le k}\text{ as just defined}\}.$$
The claim is that 
$$p_2(p_1^{-1}(J^-(A))\cap \supp\pi)=A,$$ 
where $p_1,p_2\colon M\times M\to M$ are the canonical projections onto the first and second factor respectively.
It is easy to see that $A$ is contained in $p_2(p_1^{-1}(J^-(A))\cap \supp\pi)$. More precisely let $y\in A$. Choose 
$(x,y)\in\supp\pi$. Then $x\in J^-(y)\subset J^-(A)$ since $\pi$ is causal, i.e. $y\in p_2(p_1^{-1}(J^-(A))\cap \supp\pi)$. 
For the opposite inclusion consider $y\in p_2(p_1^{-1}(J^-(A))\cap \supp\pi)$. Then there exists $x\in M$ with $(x,y)\in p_1^{-1}(J^-(A))\cap \supp\pi$, 
i.e. $(x,y)\in \supp\pi$ and $x\in J^-(A)$. 
So there exists a chain 
$$\{(u_i,w_i)\}_{0\le i\le k}\subset \supp\pi$$ 
with $u_{i+1}\in J^-(w_i)\cap \supp\mu$ for $0\le i\le k-1$, $(u_0,w_0)=(u,w)$ and $x\in J^-(w_k)$. Now define a new chain 
$$\{(u_i,w_i)\}_{0\le i\le k+1}\subset\supp\pi$$ 
identical with the original chain for $i\le k$ and 
$$(u_{k+1},w_{k+1}):=(x,y).$$ 
Since $(x,y)\in \supp\pi$ this shows that $y\in A$. 

By the marginal property one has 
$$\mu(J^-(A))= \pi(p_1^{-1}(J^-(A))\cap \supp\pi)$$ 
and 
$$\nu(A)=\pi(p_2^{-1}(A))=\pi(p_2^{-1}(p_2(p_1^{-1}(J^-(A))\cap \supp\pi)))$$
by the above characterization of $A$. Since $B\subset p_2^{-1}(p_2(B))$ for any set $B\subset M\times M$ one knows that $\nu(A)\ge \mu(J^-(A))$. 
With the inclusion $\supp \pi \subseteq J^+$ one has on the other hand $\nu(A)\le \mu(J^-(A))$, i.e. $\nu(A)= \mu(J^-(A))$.

Consequently every causal coupling $\pi'$ of $(\mu,\nu)$ has to couple $J^-(A)$ with $A$, especially the coupling guaranteed by the definition of strict 
timelikeness. But that means $J^-(A)\cap \supp\mu$ is locally uniformly bounded away from $\partial J^-(A)$. Since $J^-(A)\cap \supp\mu$ is nonempty and 
open it has to be equal to $\supp\mu$ since $\supp\mu$ is connected. This implies $A=\supp\nu$ by the construction of $A$.

Now let $x\in \supp\mu$ be given. Choose $y\in\supp\nu$ with $(x,y)\in\supp \pi$. The above argument for the set $A$ with $(u_0,w_0)=(x,y)$ yields 
that there exists a finite chain $\{(u_i,w_i)\}_{0\le i\le k+1}\subset \supp\pi$ with $c_L(u_{i+1},w_i)<\infty$ and $(u_{k+1},w_{k+1})=(x_0,y_0)$. By definition of $\psi$ one has
$$\psi(x)+ \sum_{i=0}^k [c_L(u_i,w_i)-c_L(u_{i+1},w_i)]\le \psi(x_0)=0.$$
Since 
$$\sum_{i=0}^k [c_L(u_i,w_i)-c_L(u_{i+1},w_i)]>-\infty$$ 
one obtains $\psi(x)<\infty$.

Next consider the construction of $A$ with $(u_0,w_0)=(x_0,y_0)$. Then there exists a finite chain  $\{(u'_i,w'_i)\}_{0\le i\le k+1}\subset \supp\pi$ with 
$c_L(u'_{i+1},w'_i)<\infty$ and $(u'_{k+1},w'_{k+1})=(x,y)$. This time one has 
$$\psi(x_0)+ \sum_{i=0}^k [c_L(u'_i,w'_i)-c_L(u'_{i+1},w'_i)]\le \psi(x)$$
and it follows $\psi(x)>-\infty$.
Since $c_L$ is continuous and real-valued on $J^+$, one concludes that $\psi$ is measurable. 

Define $\zeta\colon \supp\nu\to \R\cup\{-\infty\}$, 
\begin{align*}
\zeta(y):=\sup \left\{\sum_{i=0}^k[c_L(x'_i,y'_i)-c_L(x'_{i+1},y'_i)] +c_L(x'_{k+1},y)\right\},
\end{align*}
where the supremum is taken over all $k\in\N$ and sequences 
$$\{(x'_i,y'_i)\}_{0\le i\le k+1}\in\supp\pi\text{ with }y'_{k+1}=y\text{ and }(x'_0,y'_0)=(x_0,y_0).$$ 

Then one has $\psi(x)=\sup_y \{\zeta(y)-c_L(x,y)\}$, i.e. $\psi$ is $c_L$-convex. It follows, like in step 3 of the proof of \cite[Theorem 5.10 (i)]{villani}
that 
$$\psi^{c_L}(y)-\psi(x) =c_L(x,y)$$ 
for $\pi$-almost all $(x,y)\in M\times M$. This finishes the proof of Theorem \ref{thm1}.

\subsection{Proof of Theorem \ref{thm2}(i)}

Let $\Pi$ be a dynamical optimal coupling between $\mu$ and $\nu$ and $\psi\colon M\to \R\cup\{\infty\}$ be a solution to the DKP for $(\mu,\nu)$. 
The proof is carried out via contradiction, i.e. one assumes that 
$$\Pi(\Gamma_0)>0$$ 
or equivalently $\Pi|_{\Gamma_0}\neq 0$. 

One has $(\ev_0)_\sharp\Pi|_{\Gamma_0}\le \mu$ 
and therefore 
$$(\ev\nolimits_0)_\sharp\Pi|_{\Gamma_0}\ll \mathcal{L}_H.$$ 
The goal is to find a set $V$ with positive measure relative to $\mu$ such that $\psi|_V\equiv -\infty$, i.e. contradicting 
the definition of a solution to the DKP for $(\mu,\nu)$. 

Since the problem is local one can assume that $\supp\mu$ and $\supp\nu$ are compact. 
The Borel measurable map $\Gamma\to TM$, $\gamma\mapsto \dot\gamma(0)$ induces a measure  on $\partial\mathcal{C}$
via the push forward of $\Pi|_{\Gamma_0}$. 

Since $\mu$ and $\nu$ have disjoint compact support  there exists a lower bound $\e_0>0$ on the distance between points in the supports. 
In order to illuminate the construction one can, by diminishing $\e_0$ and considering an intermediate transport, assume that 
\begin{itemize}
\item[(1)] there exists a submanifold chart $U\to \R^m$ of $H\cong \R^{m-1}\times \{0\}$ such that $\partial_m\in \inte \mathcal{C}$ everywhere and 
$$\supp\mu\cup\supp\nu \subset B_{3\e_0}(0)\subset \R^m\cong U,$$ 
\item[(2)] the Riemannian metric is induced by the euclidian metric on a convex neighborhood of $\supp\mu\cup\supp\nu$ and 
\item[(3)] $\exp^\mathbb{L}$ is a diffeomorphism from an open set in $T\R^m$ onto $B_{1}(0)\times B_{1}(0)$.
\end{itemize}

In order to justify these assumptions one has to show that the intermediate transport has a solution to the DKP for the transported measures.

\begin{lemma}\label{DKP-interpolation}
Let $(\mu,\nu)\in \mathcal{P}^+_\tau(M)$ such that the DKP for $(\mu,\nu)$ has a solution and let $\Pi$ be a dynamical optimal coupling of $\mu$ and $\nu$. Further let 
$\sigma_1,\sigma_2\colon \Gamma\to [0,1]$ be measurable with $\sigma_1\le\sigma_2$. Then there exists a solution of the DKP for the intermediate transport between 
$(\ev\circ\sigma_1)_\sharp \Pi$ and $(\ev\circ\sigma_2)_\sharp \Pi$. 
\end{lemma}

\begin{proof}
(i) First consider the special case $\sigma_1\equiv 0$. The assertion then claims that for $\sigma\colon \Gamma\to [0,1]$ the DKP has a solution for the
martingales $\mu$ and $\nu_\sigma:=(\ev\circ\sigma)_\sharp \Pi$. 

Set $\pi:=(\ev_0,\ev_1)_\sharp \Pi$ and let $\psi\colon M\to \R\cup \{\infty\}$ be a solution of the DKP, i.e. $\psi|_{\supp\mu} \not\equiv \infty$ and
\begin{equation}\label{E0}
\psi^{c_L}(y)-\psi(x)=c_L(x,y),\text{ $\pi$-almost everywhere.}
\end{equation}
Choose a set $\Sigma\subset \supp \pi$ of full $\pi$-measure where \eqref{E0} is satisfied. Then $\Pi$ is concentrated on $\Sigma_\Gamma :=(\ev_0,\ev_1)^{-1}(\Sigma)$.
By definition of $\psi^{c_L}$ one has 
$$\psi^{c_L}(y)-\psi(x)\le c_L(x,y),\text{ for all }x,y\in M.$$
Assume that there exists $\gamma\in \Sigma_\Gamma$ and $t\in [0,1]$ such that 
$$\psi^{c_L}(\gamma(t))-\psi(\gamma(0))< c_L(\gamma(0),\gamma(t)).$$
Then there exists $x\in M$ with 
$$\psi(x)+c_L(x,\gamma(t))<\psi(\gamma(0))+c_L(\gamma(0),\gamma(t))$$
by the definition of the $c_L$-transform $\psi^{c_L}$. Adding $c_L(\gamma(t),\gamma(1))$ to both sides and applying the triangle inequality, which is an equality on the right hand side, one obtains
$$\psi^{c_L}(\gamma(1))\le \psi(x)+c_L(x,\gamma(1))<\psi(\gamma(0))+c_L(\gamma(0),\gamma(1)),$$
which implies 
$$\psi^{c_L}(\gamma(1))-\psi(\gamma(0))< c_L(\gamma(0),\gamma(1)),$$
a contradiction. Therefore one has 
$$\psi^{c_L}(\gamma(t))-\psi(\gamma(0))= c_L(\gamma(0),\gamma(t))$$
for all $\gamma\in \Sigma_\Gamma$ and the lemma in the special case $\sigma_1\equiv 0$ is proved.

(ii) Second consider the case $\sigma_1\equiv 0$. By definition $\psi^{c_L}$ is a $c_L$-concave function and since $\psi$ is 
$c_L$-convex one has, cf. \cite{villani},
$$\psi(x)=\sup\{\psi^{c_L}(y)-c_L(x,y)|\; y\in M\}=(\psi^{c_L})^{c_L}(x)$$
for all $x\in M$.
This implies that the transport problem between $\mu$ and $\nu$ has a solution of the DKP if and only if there exists a $c_L$-concave function 
$\phi\colon M\to \R\cup\{-\infty\}$ with 
$$\phi(y)-\phi^{c_L}(x)=c_L(x,y),\text{ $\pi$-almost everywhere}$$
for all optimal coupling $\pi$ of $\mu$ and $\nu$. With an analogous argument as in case (i) one obtains 
$$\phi(\gamma(1))-\phi^{c_L}(\gamma(t))=c_L(\gamma(t),\gamma(1))$$
for all $\gamma\in \Sigma_\Gamma$ and $t\in [0,1]$ with the notation of the special case. Setting $\phi=\psi^{c_L}$ yields the assertion.

(iii) To complete the proof consider the succession of first the intermediate transport between $\sigma_1'\equiv 0$ and $\sigma_2'\equiv \sigma$ and second 
the intermediate transport between $\sigma_1''\equiv \sigma_1$ and $\sigma_2''\equiv 1$.
\end{proof}

For $v\in T\R^m\cong TU$ denote with $v_H$ the projection of $v$ along $\partial_m$ onto span$\{\partial_1,\ldots,\partial_{m-1}\}$. 
Further let $\gamma_v$ be the unique $\Phi$-orbit with $\dot\gamma_v(0)=v$ for $v\in \mathcal{C}$.

\begin{lemma} 
For $\e_0>0$ sufficiently small there exists $C_0<\infty$ and $\phi_0>0$ such that for all $v\in \partial\mathcal{C}\cap T^1B_{3\e_0}(0)$ and $t>0$ such that 
$$\dist(\gamma_v(0),\gamma_v(t))\in [\e_0,6\e_0]$$ 
the intersection 
$$P_{v,t}:= \partial J^-(\gamma_v(t))\cap (\{\gamma_v(0)\}+\R^{m-1}\times \{0\})$$
is a smooth hypersurface in $\{\gamma_v(0)\}+\R^{m-1}\times \{0\}$ with 
\begin{itemize}
\item[(a)] the norm of the second fundamental form bounded by $C_0$ and 
\item[(b)] $\angle(v_H,TP_{v,t})>\phi_0$.
\end{itemize}
\end{lemma}

\begin{proof}
Consider for $\e>0$ the map 
$$r_\e\colon B_1(0)\to B_\e(0),\; x\mapsto \e x.$$ 
The Lagrangians $\mathbb{L}_\e:= \frac{1}{\e^2}r_\e^*\mathbb{L}$ converge uniformly  to $\mathbb{L}_0:=\mathbb{L}(0)$
on every compact subset of $TB_1(0)$ for $\e\to 0$ in every $C^k$-topology. The Euler-Lagrange flow of $\mathbb{L}_\e$ then converges uniformly on compact subsets 
in every $C^k$-topology to the Euler-Lagrange flow of $\mathbb{L}_0$. Note that the Euler-Lagrange equation for $\mathbb{L}_0$ is $\ddot x=0$. Thus one has 
$\exp_x^{\mathbb{L}_0}(\mathcal{C}_0)=x+\mathcal{C}_0$ and 
therefore the assertion (a) and (b) hold for $\mathbb{L}_0$. Since the Riemannian metric $h$ is euclidian on $B_1(0)$ one has $\frac{1}{\e^2}r_\e^*h=h$. This together 
with the fact that 
$$\exp^{\mathbb{L}_\e}_x(V_{\e,x}\cap \partial\mathcal{C}_{\e,x})=\partial J^+_\e(x)\subset B_1(0)$$ 
for some neighborhood $V_{\e,x}$ of $0_x$ in $T_x B_1(0)$ one obtains the lemma for $\e>0$ sufficiently small.
\end{proof}

Statement $(\mathrm{b})$ is equivalent to requiring 
$$\angle (v_H,N_{v,t})\le \frac{\pi}{2}-\phi_0$$ 
where $\angle(v,w)$ denotes the euclidian angle between $v$ and $w$ and $N_{v,t}$ denotes the inward pointing unit normal to 
$P_{v,t}$. Property (a) implies that for $\e_1:=\frac{1}{\sqrt{C_0}}>0$ one has 
$$B_{\e_1}(\gamma_v(0)+\e_1N_{v,t})\cap (\{\gamma_v(0)\}+\R^{m-1}\times \{0\}) \subset J^-(\gamma_v(t))\cap (\{\gamma_v(0)\}+\R^{m-1}\times \{0\}).$$ 
This and (b) then imply that 
$$\{\gamma_v(0)\}+\text{Cone}\left(v,\frac{\phi_0}{2},\e_2\right)\subset J^-(\gamma_v(t))\cap (\{\gamma_v(0)\}+\R^{m-1}\times \{0\})$$
where 
$$\text{Cone}(v,\phi,\e):=\{w\in \R^{m-1}\times\{0\}|\angle(v_H,w)<\phi, |w|<\e\}$$
and
$$\e_2:=2\e_1\sin\frac{\phi_0}{2}.$$

Now for $p\in B_{3\e_0}(0)$, $v\in \partial \mathcal{C}_p$, $q\in \{p\}-\text{Cone}(v,\frac{\phi_0}{4},\e_2)$ and $w\in \partial 
\mathcal{C}_q$ with $\angle(w_H,v_H)<\frac{\phi_0}{4}$ one has 
$$B_{\dist(p,q)\sin\frac{\phi_0}{4}}(p)\cap (\{\gamma_v(0)\}+\R^{m-1}\times \{0\})\subset J^-(\gamma_w(t))$$
for $\dist(q,\gamma_w(t))\in [\e_0,3\e_0]$. 

Abbreviate 
$$A:=\supp [(\ev\nolimits_0)_\sharp(\Pi|_{\Gamma_0})].$$
Define a map 
$$T\colon \supp (\Pi|_{\Gamma_0})\to S^{m-2},\; \gamma \mapsto \dot\gamma(0)_H.$$
Since $(\pi_{TM}\circ T)_\sharp \Pi|_{\Gamma_0}=\mu|_A$ and $\mu(A)\neq 0$ one has 
$$T_\sharp (\Pi|_{\Gamma_0})\neq 0$$
where $\pi_{TM}\colon TM\to M$ denotes the canonical projection. Choose $v_0\in S^{m-2}$ such that 
$$T_\sharp (\Pi|_{\Gamma_0})(B_{\frac{\phi_0}{4}}(v_0))>0.$$ 
Then $A':= \pi_{TM}(B_{\frac{\phi_0}{4}}(v_0))$ has positive measure with respect to $\mathcal{L}_H$, since $\mu|_A\ll \mathcal{L}_H$ . 
Let $p$ be a Lebesgue point of $A'$, i.e. 
$$\lim\nolimits_{\delta\downarrow 0}\frac{\mathcal{L}_H(A'\cap B_\delta(p))}{\mathcal{L}_H(B_\delta(p))}=1.$$
Then there exist $\e_3,\e_4>0$ such that for the unique $v\in\partial\mathcal{C}$ with $v_H=v_0$
$$\mathcal{L}_H\left(A'\cap B_{2\e_3}(p)\setminus B_{\e_3}(p)\cap \left(\{p\}-\text{Cone}\left(v,\frac{\phi_0}{4},\e_2\right)\right)\right)\ge \e_4.$$
Choose polar coordinates $(r,\theta_1,\ldots ,\theta_{m-2})$ on $B_{2\e_3}(p))\subset \R^{m-1}\times \{0\}$ centered at $p$. By Fubini's Theorem there 
exists $(\eta_1,\ldots ,\eta_{m-2})$ such that 
$$\int_{\e_3}^{2\e_3} \chi_{A'}(r,\eta_1,\ldots ,\eta_{m-2}) rdr\ge \e_4.$$ 
Thus the $1$-dimensional Lebesgue measure 
$$\mathcal{L}_1(A'\cap \{r\in(\e_3,2\e_3),\theta_1=\eta_1, \ldots ,\theta_{m-2}=\eta_{m-2}\})\ge \frac{\e_4}{2\e_3}.$$

Recall Lemma \ref{lemmadots}:
\begin{lem}
Let $[a,b]\subset \R$, $\e>0$ and $B\subset [a,b]$ Borel measurable be given with $\mathcal{L}_1(B)\ge \e(b-a)$. Then for all $n\in\N$ there exists 
$\{t_i\}_{1\le i\le n}\subset B$ with $t_1<\ldots <t_n$ and $t_{i+1}-t_i\ge \frac{\e}{2n}$.
\end{lem}
Applying Lemma \ref{lemmadots} to 
$$B:=A'\cap \{r\in(\e_3,2\e_3),\theta_1=\eta_1, \ldots ,\theta_{m-2}=\eta_{m-2}\}$$ 
yields for every $n\in\N$ points 
$$x_1,\ldots x_n\in A'\cap \{r\in(\e_3,2\e_3),\theta_1=\eta_1, \ldots ,\theta_{m-2}=\eta_{m-2}\}$$ 
and $y_1,\dots,y_n\in M$ such that $(x_i,y_i)\in\supp \pi$ for the optimal coupling 
$\pi$ induced by $\Pi$, $x_{i+1}\in J^-(y_i)$, $\dist(x_i,x_{i+1})\ge \frac{\e_4}{4n\e_3}$ and $B_{\e_3\sin \frac{\phi_0}{4}}(p)\subset J^-(y_n)$. 
Using the exponential map of $\mathbb{L}$ one sees that there exists $C>0$ such that $c_L(x_{i+1},y_i)\le -\frac{C}{\sqrt{n}}$. Then one has for every point 
$x'\in B_{\e_3\sin \frac{\phi_0}{4}}(p)$
$$\psi(x')\le \psi(x_1)+\sum c_L(x_{i+1},y_i)\le \psi(x_1)-(n-1)\frac{C}{\sqrt{n}}\le \psi(x_1)-\frac{C}{2}\sqrt{n}$$
for all $n\in \N$ where $\psi$ denotes the solution to the DKP for $(\mu,\nu)$. Thus 
$$\psi|_{B_{\e_3\sin \frac{\phi_0}{4}}(p)}\equiv -\infty$$ 
therefore contradicting that $\psi$ is $c_L$-convex.

\subsection{Proof of Theorem \ref{thm3}(ii)}
One can prove Theorem \ref{thm3}(ii) in the same fashion as Theorem \ref{thm2}(i), but for the sake of avoiding repetition Theorem \ref{thm3}(ii) is reduced to 
Theorem \ref{thm2}(i).

Let  $f_\mu\in L^1(\mathcal{L})$ be the density of $\mu$ with respect to the Lebesgue measure $\mathcal{L}$ on $M$. 
If $\Pi(\Gamma_0)>0$ there exists a chart $U\to \R^m$ of $M$ such that $U$ is foliated by spacelike hypersurfaces
$\{H_s=\R^{m-1}\times\{s\}\}_{s\in\R}$ and $(\ev_0)_\sharp \Pi|_{\Gamma_0}(U)>0$. By \cite[Corollary 3.5]{suhr16} one can then assume that $\supp\mu$ is contained in $U$. 
By Fubini's Theorem the restriction of $f_\mu$ to $H_s$ is integrable with respect to the Lebesgue measure $\mathcal{L}_{m-1}$ on $\R^{m-1}$ for almost 
all $s\in\R$. For $(x_m)_\sharp\mu$-almost all $s\in \R$ one has 
$$m_s:= \int_{H_s}f_\mu d\mathcal{L}_{m-1} >0$$ 
where $x_m$ denotes the $m$-th
coordinate function on $\R^m$. For $s$ with $m_s\in (0,\infty)$ define $\mu_s$ by 
$$\mu_s(A):=\frac{1}{m_s}\int_{A\cap H_s} f_\mu d\mathcal{L}_{m-1}.$$ 
Note that $\mu_s$ is a probability measure on $H_s$. Consider the disintegration $\{\Pi_s\}_{s\in\R}$ of $\Pi$ along $X_m:=x_m\circ \ev_0\colon \Gamma\to\R$. 
Then one has $(\ev_0)_\sharp \Pi_s=\mu_s$ for $\mu_m$-almost all $s$. This can be seen as follows: Obviously one has $(X_m)_\sharp \Pi=\mu_m$ 
since $(\ev_0)_\sharp \Pi= \mu$. It follows that for all Borel measurable $B\subset M$ one has:
\begin{align*}
&\int_\R (\ev\nolimits_0)_\sharp \Pi_s(B)d\mu_m=\int_\R \Pi_s(\ev_0\nolimits^{-1}(B))d\mu_m\\
&=\int_\R \Pi_s(\ev_0\nolimits^{-1}(B))d(X_m)_\sharp \Pi=\Pi(\ev_0\nolimits^{-1}(B))=\mu(B)
\end{align*}
Thus one has 
$$\mu(.)=\int_\R (\ev\nolimits_0)_\sharp \Pi_s(.)d\mu_m$$ 
and the claim follows from the uniqueness part of the Disintegration Theorem. Finally define $\nu_s:=(\ev_1)_\sharp \Pi_s$. 

\begin{lemma}
$\Pi_s$ is an optimal dynamical coupling of $\mu_s$ and $\nu_s$ for $\mu_m:=(x_m)_\sharp\mu$-almost all $s\in\R$. Further $\psi$ solves the DKP for 
$(\mu_s,\nu_s)$ for $\mu_m$-almost all $s$.
\end{lemma}

\begin{proof}
The first assertion follows from the second since the pair $(\psi,\psi^{c_L})$ is admissible (see Section \ref{results}). Define $\pi_s:=(\ev_0,\ev_1)_\sharp \Pi_s$. Assume that $\psi$ does not 
solve the DKP for $(\mu_s,\nu_s)$ for $s$ in a set $B\subset \R$ of positive $\mu_m$-measure, i.e. 
$$\int_{M\times M} c_L(x,y)d\pi_s >\int_M \psi^{c_L}(y)d\nu_s -\int_M \psi(x) d\mu_s$$ 
for all $s\in B$. Then for some $\delta>0$ one has 
$$\int_{M\times M} c_L(x,y)d\pi_s -\delta>\int_M \psi^{c_L}(y)d\nu_s -\int_M \psi(x) d\mu_s$$
for $s$ in a smaller but $\mu_m$-non-negligible set $B_\delta\subset \R$. This implies that 
$$\int_{M\times M}c_L(x,y) d\pi -\delta \int_{B_\delta}m_s d\mu_m> \int_M \psi^{c_L}(y)d\nu -\int_M \psi(x) d\mu$$
since $c_L(x,y)\ge \psi^{c_L}(y)-\psi(x)$ for all $(x,y)$. Note that $\int_{A_\e} m_sd\mu_m>0$ since $B_\delta$ is $\mu_m$-non-negligible.  This contradicts the 
assumption that $\psi$ is a solution to the dual problem for $(\mu,\nu)$ since $\pi$ is minimal.
\end{proof}

Theorem \ref{thm2}(i) then yields that for almost all $s\in \R$ one has $\Pi_s(\Gamma_0)=0$. Since 
$$\Pi(\Gamma_0)=\int_\R m_s\Pi_s(\Gamma_0) ds$$ 
the claim follows. The last equation follows from the respective statement about $\mu_s$, i.e. 
$$\mu(A)=\int_\R m_s \mu_s(A)ds$$ 
for all measurable $A\subset \R^m$.

\subsection{Proof of Theorem \ref{thm4}}

The proof of Theorem \ref{thm4} relies on ideas of \cite{CH2015NBTrans}, see also 
\cite{Gigli2012,CM2016TransMapsMCP,kell17}. Note, however, there is no unique equivalent 
to the assumption of achronality resp. acausality of the support of the second measure. 
Indeed, the Monge problem is in general highly non-unique. However, the proof in the 
non-relativistic as well as the relativistic setting relies essentially on the following
two properties: Geodesics with endpoints in a given set are non-branching  
(Lemma \ref{lem:strong non-branching}) and that there is a weak form of the measure contraction property (Lemma \ref{L1}). 
The latter holds due to differentiability of the time function and the exponential map. 

For the proof it suffices to consider the case that both $\supp\mu$ and $\supp\nu$ are compact and disjoint. This follows from the observation that absolute continuity 
is equivalent to absolute continuity on every compact subset.

Recall that
$$\Phi^\mathbb{L}\colon \mathbb{U}\subset \R\times TM\to TM$$ 
denotes the Euler-Lagrange flow of 
$\mathbb{L}$. Set 
$$\{1\}\times \mathbb{V}:=(\{1\}\times TM)\cap \mathbb{U}.$$
By \cite[Proposition 3.14]{suhr16} the map 
$$\exp^\mathbb{L}\colon\mathbb{V}\to M\times M, \; v\mapsto (\pi_{TM}(v),\pi_{TM}\circ \Phi^\mathbb{L}(1,v))$$
is a $C^1$-diffeomorphism of a neighborhood of the zero section onto its image and smooth outside $T^0M$. Here 
$$\pi_{TM}\colon TM\to M$$ 
denotes the canonical projection. 
Set 
$$\exp_x\colon \mathbb{V}\cap TM_x\to M,\; v\mapsto \pi_{TM}\circ \Phi^{\mathbb{L}}(1,v).$$
Further since $\partial_t (\pi_{TM}\circ \Phi^\mathbb{L}(t,v))=\Phi^\mathbb{L}(t,v)$ and $\Phi^\mathbb{L}(1,tv)=\Phi^\mathbb{L}(t,v)$ one has 
\begin{equation}\label{E2}
d(\exp_x)_0= \partial_v (\pi_{TM}\circ \Phi^\mathbb{L}(1,v))|_{v=0} =\id\nolimits_{TM}
\end{equation}
via the canonical identification $TM_x\cong T(TM_x)_0$.

\begin{prop}\label{P1a}
Let $v\in \mathcal{C}_p:=\mathcal{C}\cap TM_p$ such that 
$$d(\exp_p)_v\colon T(TM_p)_v\to TM_{\exp_p^\mathbb{L}(v)}$$
is singular. Then for every $T>1$ the geodesic 
$$\eta_v\colon [0,T]\to M,\; t\mapsto \exp_p(tv)$$
is not $\mathcal{A}$-minimizing between its endpoints.
\end{prop}

\begin{proof}
In the case of $v\in \inte\mathcal{C}$ the claim follows mutatis mutandis as in \cite[Proposition 7.4.1.]{bcc}, since $C^1$-small variations of timelike curves
remain timelike.

The case $v\in \partial\mathcal{C}$ is the subject of \cite[Proposition 6.8]{aj}. Note that the definition of Lorentz-Finsler metrics therein is equivalent to the 
presently used by virtue of \cite{minguzzi16}. 
\end{proof}

\begin{cor}\label{C1a}
Let $K\subset M$ be compact. Consider the set $\mathcal{K}$ of $\mathcal{A}$-minimal causal $\mathbb{L}$-geodesics $\eta\colon [0,1]\to M$, i.e. $\eta(t)=
\exp_{\eta(0)}(t\dot\eta(0))$ and $\mathcal{A}(\eta)=c_L(\eta(0),\eta(1))$, with $\eta(0),\eta(1)\in K$. Then there exists a continuous function
$e\colon (0,1)\to \R_{>0}$ with 
$$\|d(\exp_{\eta(0)})_{t\dot\eta(0)}^{-1}\|\le e(t)$$
for all $\eta\in \mathcal{K}$ and $t\in [0,1)$. Further one has $\lim_{t\to 0}e(t)=1$.
\end{cor}

\begin{proof}
For a single $\mathcal{A}$-minimal geodesic $\eta\colon [0,1]\to M$ the claim follows directly from Proposition \ref{P1a}. Further by \eqref{E2} 
one has $e(t)\to 1$ for $t\to 0$.

Since $K$ is compact the corollary follows from the continuity of the differential of the exponential map.
\end{proof}

For $v\in \mathcal{C}$ consider the unique time affinely parameterized local $\mathcal{A}$-minimizer $\gamma_v\colon \R\to M$ with $\dot\gamma_v(0)=v$.
According to \cite[Section 3.5]{suhr16} $\gamma_v$ is an orbit of a flow on $\mathcal{C}\setminus T^0M$. Define the map 
$$\exp^\tau \colon \mathcal{C}\to M, \; v\mapsto \gamma_v(1).$$
Denote with $\exp^\tau_p$ the restriction of $\exp^\tau$ to $\mathcal{C}\cap TM_p$. The map satisfies $\exp^\tau(tv)=\gamma_v(t)$ which implies 
\begin{equation}\label{E2a}
d(\exp^\tau_p)_v\to \id\nolimits_{TM_p}
\end{equation}
for $v\to 0$ and
\begin{equation}\label{E2b}
d(\exp^\tau_p)_v(v)=\dot\gamma_v(1).
\end{equation}
Further by \cite[Proposition 3.5]{suhr16} there exists a smooth map 
$$s\colon \mathcal{C}\setminus T^0M\to [0,\infty)$$
such that $\exp^\tau(v)=\exp(s(v)v)$ for all $v\in \mathcal{C}\setminus T^0M$. One has
\begin{equation}\label{E2c}
d(\exp^\tau_p)_v=d(\exp_p)_{s(v)v}(ds_v\otimes v+s(v)\cdot \id\nolimits_{TM_p}).
\end{equation}

The following lemma is at the heart of the proof of Theorem \ref{thm4}. It is an easy consequence of differentiability
of the time function $\tau$ and the exponential map.

\begin{lemma}[weak measure contraction property]\label{L1}
Let $K\subset M$ be compact. Then there exists a function $f\colon(0,1)\to(0,1)$ with $\lim_{t\to 0} f(t)=1$ such that for all $y\in K$
and all measurable $A\subset K\cap J^-(y)$ it holds 
\[
\mathcal{L}(A_{t,y})\ge f(t)\mathcal{L}(A)
\]
where $A_{t,y}=\{ev_{t}\gamma\,|\,\gamma\in(ev_{0},ev_{1})^{-1}(A\times\{y\})\}$.
\end{lemma}

\begin{proof}
From \eqref{E2b} and \eqref{E2c} follows that $d(\exp^\tau_p)_v$ is singular if and only if $d(\exp_p)_v$ is singular. By Corollary \ref{C1a}
and continuity of both $d\exp^\tau_p$ and $d\exp_p$ there exists a function $e^\tau \colon (0,1)\to \R_{>0}$ with 
\begin{equation}\label{E2d}
\|d(\exp^\tau_{\gamma(0)})_{t\dot\gamma(0)}^{-1}\|\le e^\tau(t)
\end{equation}
for any $\gamma\in (\ev_0,\ev_1)^{-1}(K\times \{y\})$. Now consider 
$$T_{K,y}:=\{v\in \mathcal{C}|\; \exists \gamma \in (\ev\nolimits_0,\ev\nolimits_1)^{-1}(K\times \{y\})\text{ with }\dot\gamma(0)=v\}.$$
It follows that 
$$A_{t,y}=\exp^\tau(t(T_{K,y}\cap (\pi_{TM})^{-1}(A))).$$
By equation \eqref{E2d} there exists a function $f\colon (0,1)\to (0,1)$ independent of $A$ with 
\[
\mathcal{L}(A_{t,y})\ge f(t)\mathcal{L}(A)
\]
for all $t\in (0,1)$. With property \eqref{E2a} one concludes $f(t)\to 1$ for $t\to 0$.
\end{proof}

For the following proposition observe that due to the fact that geodesics $\gamma$ with $(\tau\circ\gamma)'\equiv\mathsf{const}$ are uniquely defined by their initial 
velocity one knows that the image of such a geodesic is a one-dimensional rectifiable curve. In particular, 
it has zero measure with respect to the Lebesgue measure $\mathcal{L}$ on $M$. This implies that for distinct points $x,y\in M$ the set 
\begin{align*}
B_{x,y} := \{z\in J^-(y)\cap J^-(x) \,|\, &c_L(z,y)=c_L(z,x)+c_L(x,y)\,\\ &\textbf{or }\,c_L(z,x)=c_L(z,y)+c_L(y,x)\}
\end{align*}
has vanishing $\mathcal{L}$-measure. In particular, 
$$\mathcal{L}(A_{t,x} \cap A_{t,y})=0$$
for $x\neq y$. Note that if $x$ and $y$ are not causally related then  
$A_{t,x} \cap A_{t,y} = \varnothing$.

A more general statement of this form was obtained by the second author in \cite{suhr16}. 

\begin{lemma}\label{L20}
Let $B$ be a closed achronal set. Then the set 
\[
\bigcup_{x\ne y \in B} B_{x,y}
\]
has vanishing $\mathcal{L}$-measure.
\end{lemma}

\begin{proof}
For $z\in \cup_{x\neq y\in B} B_{x,y}$ there exist $x,y\in B$ with $z\in B_{x,y}$. Especially one has $c_L(z,x),c_L(z,y)<\infty$. By the definition of $B_{x,y}$ one concludes
that $c_L(x,y)<\infty$ or $c_L(y,x)<\infty$. 

Now let $\Gamma_B\subset \Gamma$ be the subspace of minimizers which intersect $B$ at least twice. By \cite[Proposition 3.22]{suhr16} the set 
$$\bigcup_{\gamma\in \Gamma_B} \gamma(\R)\cap\{\tau=r\}$$
has vanishing Lebesgue measure in $\{\tau=r\}$ for all $r\in\R$. Thus 
$$\bigcup_{\gamma\in \Gamma_B} \gamma(\R)$$
has vanishing Lebesgue measure in $M$. Since 
$$\bigcup_{x\ne y \in B} B_{x,y}\subset \bigcup_{\gamma\in \Gamma_B} \gamma(\R)$$
the claim follows.
\end{proof}

For a map $\sigma:\Gamma\to[0,1]$ and a geodesic $\gamma \in \Gamma$ one writes $\gamma_\sigma = \gamma(\sigma(\gamma))$. 
Let $S \subset \Gamma$ be a subset of the space of minimizing geodesics. Then for $s,t\in[0,1]$ one
defines 
$$S_{t,s} := (\ev\nolimits_s,\ev\nolimits_t)(S)\subset M\times M,$$
 $S_t := \ev_t(S)$ and $S_\sigma:=\{\gamma_\sigma\}_{\gamma\in S}\subset M$.

\begin{lemma}\label{lem:strong non-branching}
Let $S\subset \Gamma$ be such that $S_{0,1}$ is $c_L$-cyclically monotone. 
If $S_1$ is achronal then for all Borel-measurable maps $\sigma:S \to (0,1)$  one has 
$$\mathcal{L}(S^{(1)}_\sigma \cap S^{(2)}_\sigma)=0$$ 
for all Borel-measurable $S^{(1)},S^{(2)}\subset S$ with $S^{(1)}_1 \cap S^{(2)}_1 = \varnothing$.
\end{lemma}

\begin{remark}
If $S_1$ is additionally acausal, e.g. $S_1$ is contained in a time-slice $\{\tau=\tau_0\}$, then one even has $S^{1}_\sigma \cap S^{2}_\sigma = \varnothing$. 
\end{remark}
\begin{proof}
Let $z \in S^{1}_\sigma \cap S^{2}_\sigma$. Then there exist $\gamma_1 \in S^{(1)}$ and $\gamma_2 \in S^{(2)}$ such that 
$$c_L(x_i,y_i)=c_L(x_i,z)+c_L(z,y_i)$$
for $i=1,2$, $x_i:=\gamma_i(0)$ and $y_i:=\gamma_i(1)$.

Since $S_{0,1}$ is $c_L$-cyclically monotone one has
\begin{align*}
c_{L}(x_{1},y_{1})+c_{L}(x_{2},y_{2}) & \le c_{L}(x_{2},y_{1})+c_{L}(x_{1},y_{2})\\
 & \le(c_{L}(x_{2},z)+c_{L}(z,y_{1}))+(c_{L}(x_{1},z)+c_{L}(z,y_{2}))\\
 & =c_{L}(x_{1},y_{1})+c_{L}(x_{2},y_{2}).
\end{align*}
This shows that there is a minimizing geodesic connecting $x_1$ with $y_2$ and passing through $z$. As geodesics are locally unique, 
either $y_1$ is on the geodesic connecting $z$ and $y_2$ or $y_2$ is on the geodesic
connecting $z$ and $y_1$. Thus $z \in B_{y_1,y_2}$ with $y_1 \ne y_2$.

In particular, 
$$S^{1}_\sigma \cap S^{2}_\sigma \subset \bigcup_{y\ne y' \in S_1} B_{y,y'}.$$
Since $S_1$ is achronal it follows that 
$$\mathcal{L}(S^{1}_\sigma \cap S^{2}_\sigma)=0$$
by Lemma \ref{L20}.
\end{proof}

Now one combines Lemma \ref{lem:strong non-branching} with the weak measure contraction property to obtain the following.
\begin{prop}\label{P2}
Assume that
$$A\times\{y,z\}\subset\operatorname{supp}\mu\times\operatorname{supp}\nu$$
is $c_{L}$-cyclically monotone for an achronal two-point set $\{y, z\}$ and some measurable
set $A$. Then $A$ has vanishing Lebesgue measure.
\end{prop}
\begin{proof}

By inner regularity of $\mathcal{L}$ one may assume $A$ is compact so that
for a fixed $\epsilon>0$ and $t$ sufficiently close to $0$ it holds
\[
A_{t,y}\cup A_{t,z}\subset A_{\epsilon}
\]
where $A_{\epsilon}$ is the $\epsilon$-neighborhood of $A$ with respect to the distance $dist$. Lemma \ref{lem:strong non-branching} implies that 
\[
\mathcal{L}(A_{t,y}\cap A_{t,z})=0\quad\text{for all }t\in[0,1).
\]
Then the weak measure contraction property yields
\begin{align*}
\mathcal{L}(A) & =\lim_{\epsilon\to0}\mathcal{L}(A_{\epsilon})\\
 & \ge\limsup_{t\to0}\mathcal{L}(A_{t,y}\cup A_{t,z})\\
 & =\limsup_{t\to0}\mathcal{L}(A_{t,y})+\mathcal{L}(A_{t,z})\\
 & \ge2\limsup_{t\to0}f(t)\mathcal{L}(A)=2\mathcal{L}(A)
\end{align*}
which can hold only if $A$ has zero measure.
\end{proof}

Lemma \ref{lem:strong non-branching} can be used to prove an interpolation inequality in form of the weak measure contraction 
property between any absolutely continuous measure and a causally related achronal discrete measure. In order
to prove such an interpolation inequality for general achronal target measures one needs to approximate the target measures via
finite measures which satisfy the achronality assumption. As such an approximation seems difficult, one proceeds in two steps: 
As measures supported in a time slice can be easily approximated one first proves the interpolation 
inequality for those measures. In a second step one uses this fact together with the strong non-branching property 
implied by Lemma \ref{lem:strong non-branching} to approximate general achronal target measures.

Given a subset $C\subset M\times M$ and $s,t\in (0,1)$ define
\[
C_{s,t}=\{(\ev\nolimits_s\gamma,\ev\nolimits_{t}\gamma)\,|\,\gamma\in(ev_{0},ev_{1})^{-1}(C)\}
\]
and $C_t=p_1(C_{t,t})$.

\begin{lemma}\label{lem:finite-approximation}
Assume $\pi$ is an optimal coupling with compact support between an absolutely continuous measure $\mu$ and a measure $\nu$ with 
support in a time-slice $\{\tau = \tau_0\}$. Then there is a sequence $\nu_n = \sum_{i=1}^{N_n}\lambda^n_i \delta_{x_i^n}$
such that $\operatorname{supp}\nu_n \subset \{\tau=\tau_0\}$ and the optimal couplings $\pi_n$ of $(\mu,\nu_n)$
converge weakly to an optimal coupling $\pi'$ of $(\mu,\nu)$. 
\end{lemma}

\begin{proof}
Let $X\in\Gamma(TM)$ be a vector field with $L(X)<0$ and $d\tau(X)=1$ everywhere and denote with $\phi_t$ the flow of $X$.
The coupling 
$$\pi_\e':=(\id,\phi_\e)_\sharp \pi$$ 
is supported in $\{c_L<0\}$. Recall that by \cite[Proposition 3.10]{suhr16} there exists a Borel map $S\colon J^+\to C^0([0,1],M)$ with $S(x,y)\in \Gamma_{x\to y}$. 
The push forward $\Pi_\e':=S_\sharp \pi_\e'$ is then a dynamical coupling. Consider the subset $\Gamma_{t_0}\subset \Gamma$ of minimizers $\gamma$ with 
$\tau\circ \gamma(0)\le t_0$ and $\tau\circ \gamma(1)\ge t_0+\e$.
The maps 
$$T_0\colon \Gamma_{t_0}\to [0,1],\; \gamma\mapsto T_0(\gamma)$$
such that $\tau(\gamma(T_0(\gamma)))=t_0$ and 
$$R\colon \Gamma_{t_0}\to \Gamma_{t_0},\; \gamma\mapsto [t\mapsto \gamma(T_0(\gamma)t)]$$
are continuous. Define 
$$\Pi_\e:=R_\sharp (\Pi_\e')\text{ and }\pi_\e:=(\ev\nolimits_0,\ev\nolimits_1)_\sharp \Pi_\e.$$
It then follows that $\supp \pi_\e \subset \{c_L<0\}$ and $\supp (p_2)_\sharp \pi_\e\subset \{\tau=\tau_0\}$.
Furthermore, the Prokhorov distance between $\nu_\epsilon:=(p_2)_\sharp \pi_\epsilon$ and $\nu$ tends to zero for $\e\to 0$. 

Observe that for any approximation by finite measures $(\nu_{n,\epsilon})$ of $\nu_\epsilon$ the $C_L$-cost between $\mu$ and $\nu_{n,\epsilon}$ 
is eventually finite and the distance between $\nu$ and $\nu_n$ is eventually less than $2\epsilon$. One may also assume
that $(\nu_{n,\epsilon})$ has support in $\{\tau=\tau_0\}$.

Denote by $\pi_{n,\epsilon}$ the $c_L$-optimal coupling of $(\mu,\nu_{n,\epsilon})$. 
Then 
$$\liminf_{n\to\infty}\int c_{L}d\pi_{n,\epsilon}\le\int c_{L}d\pi_{\epsilon}.$$ 
To conclude just observe that for a diagonal sequence $\pi_{(k)}=\pi_{n_k,\frac{1}{k}}$ 
one has $\pi_{(k)}\rightharpoonup \tilde{\pi}$ satisfying 
$$\int c_L d\tilde\pi \le \liminf_{k\to\infty} \int c_L d\pi_{(k)} \le \int c_L d\pi.$$ 
Since $\pi$ is optimal the $\tilde\pi$ must be optimal as well.
\end{proof}

\begin{remark}
If an optimal coupling $\pi$ is supported in the interior of $J^+$ then it is possible to obtain an approximation $\pi_n$ 
with finite target measures which have support in $\operatorname{supp}\left((p_2)_\sharp \pi\right)$. Thus it follows that it is possible
to keep the target approximation $\nu_n$ in a fixed achronal set $B$. Note, however, such an approximation
for purely lightlike optimal couplings is not always possible. It even seems difficult to prove Lemma \ref{lem:finite-approximation} under the 
assumption that $\supp\nu$ is achronal.
\end{remark}

\begin{prop}\label{P3}
Let $(\mu,\nu)\in \mathcal{P}^+_\tau(M)$ with $\mu$ being absolutely continuous and 
$\operatorname{supp}\nu \subset \{\tau=\tau_0\}$.
Then there is an optimal coupling $\pi$ of $(\mu,\nu)$ such that for all $c_L$-cyclically monotone sets $C\subset \operatorname{supp} \pi$ 
with $\pi(C)=1$ it holds 
\[
\mathcal{L}(C_{t})\ge f(t)\mathcal{L}(C_{0}).
\]
\end{prop}
\begin{proof}

First note if $\nu$ is a finite measure then the support $C=\operatorname{supp}\pi$ of any $c_L$-cyclically monotone
coupling $\pi$ satisfies the assumption. Indeed, the set of points $x \in C_0$ such that $(x,y),(x,y')\in C$ for 
distinct $y,y'\in C_1$ has zero $\mathcal{L}$-measure, i.e. $\mathcal{L}(A^i_{t,y_i} \cap A^j_{t,y_j})=0$
for $i\ne j$ and $t\in [0,1)$ where $\{y_i\}_{i=1}^n=C_1$ and $A^i = p_1((p_2)^{-1}(y_i))$. Observe now
\begin{align*}
\mathcal{L}(C_{t}) =\sum_{i=1}^{n}\mathcal{L}(A_{t,y_{i}}^{i})\ge\sum_{i=1}^{n}f(t)\mathcal{L}(A^{i})=\mathcal{L}(C_{0})
\end{align*}

For more general $\nu$ let $\pi$ be the weak limit of a sequence $\pi_n$ with $(p_2)_\sharp \pi_n$ finite 
as given by Lemma \ref{lem:finite-approximation}. Note by restricting the first marginal of $\pi_n$ slightly 
one can assume that the support of $\pi_n$ converges in the Hausdorff metric to the support of $\pi$. Note
that since $\pi_n$ converges weakly to $\pi$ one must have $\mu(C^{n}_0)\to 1$ where $C^{n}=\operatorname{supp}\pi_n$.

If $C=\operatorname{supp}\pi$ is $c_L$-cyclically monotone 
then for all $\epsilon>0$ and for sufficiently large $n\in\N$ it holds $(C^{(n)}_t) \subset (C_t)_\epsilon$. 
Since $C_t$ is compact and $C^{(n)}_0 = C_0$ one obtains
\begin{align*}
\mathcal{L}(C_{t}) & = \lim_{\epsilon\to 0} \mathcal{L}((C_t)_\epsilon) \\
                       & = \limsup_{n\to \infty} \mathcal{L}(C^{(n)}_t) \\
                       & \ge \limsup_{n\to \infty} f(t)\mathcal{L}(C^{(n)_t}) \\
                       & = f(t) \mathcal{L}(C_0). 
\end{align*}

If the support of $\pi$ is not $c_L$-cyclically monotone, one may find a $c_L$-cyclically monotone 
subset $C\subset \operatorname{supp}\pi$ of full $\pi$-measure and compact sets 
$C^{k}\subset C$ such that $\pi(C^{k})\to \pi(C)$ and $\mathcal{L}(C^{k}_0) \to  \mathcal{L}(C_0)$.

Denote by $\pi_k$ the coupling obtained by restricting $\pi$ to $C^{k}$ and renormalizing. 
Note that each $\pi_k$ is supported in $C^{k}$ and is given as a weak limit of an appropriate restriction of the
approximating sequence $\pi_n$. In particular, one sees that the claim of the proposition holds for $C^{k}$ so 
that one concludes with the following chain of inequalities
\begin{align*}
\mathcal{L}(C_{t}) & \ge\limsup_{k\to\infty}\mathcal{L}(C_{t}^{k})\\
 & \ge\limsup_{k\to\infty}f(t)\mathcal{L}(C_{0}^{k})=f(t)\mathcal{L}(C_{0}).
\end{align*}
\end{proof}

Combining the results above one obtains the existence and uniqueness of optimal transport maps
if the target is supported in a time-slice.

\begin{prop}\label{P20}
Between any absolutely continuous probability measure $\mu$ and any probability measure $\nu$ supported in a time-slice $\{\tau=\tau_0\}$ 
such that $(\mu,\nu)\in \mathcal{P}^+_\tau(M)$
there exists a unique $c_L$-optimal coupling $\pi$ and this coupling is induced by a transport map. 
\end{prop}

\begin{proof}

Let $\pi$ be an optimal coupling for $(\mu,\nu)$ and choose a $c_L$-cyclically monotone
measurable set $C \subset \operatorname{supp}\pi$ of full $\pi$-measure.

We claim $\pi$ is induced by a transport map. Note that this implies that $\pi$ is unique.

Suppose the statement was wrong. Then the Selection Dichotomy in \cite[Theorem 2.3]{kell17} 
gives couplings $\pi_1,\pi_2\ll \pi$ which are supported on disjoint sets $K\times A_1$ and $K\times A_2 $
and their first marginals are equal to $\mu_K = \frac{1}{\mu(K)}\mu\big|_K$, where $K\subset M$ is compact. Since $\mu$ is absolutely continuous
one can additionally assume $\mu_K$ and $\mathcal{L}\big|_K$ are mutually absolutely continuous.

It is easy to see that all three measures $\pi_1$, $\pi_2$ and $\frac{1}{2}(\pi_1 + \pi_2)$ are optimal. Thus by Proposition \ref{P3}
there are optimal couplings $\tilde{\pi}_i$ between $(\mu_K, (p_2)_\sharp \pi_i)$ such that 
the couplings $\tilde{\pi}_i$ are concentrated on disjoint $c_L$-cyclically monotone sets $C^{i}$ satisfying  
\begin{align*}
\mathcal{L}(C_{t}^{i}) & \ge f(t)\mathcal{L}(C^{i}).
\end{align*}

Let $\epsilon>1$. Then $C^i_t \subset K_\epsilon$ for sufficiently small $t$ where $K_\epsilon$ denotes the $\epsilon$-neighborhood of $K$. Since the sets 
$C^{1}_0$ and $C^{2}_0$ are disjoint by Lemma \ref{lem:strong non-branching} the sets 
$C_{t}^{1}$ and $C_{t}^{2}$ are disjoint as well so that one obtains
\begin{align*}
\mathcal{L}(K) & =\lim_{\epsilon\to0}\mathcal{L}(K_{\epsilon})\\
 & \ge\limsup_{t\to0}\mathcal{L}(C_{t}^{1}\dot{\cup}\,C_{t}^{2})\\
 & =\limsup_{t\to0}\mathcal{L}(C_{t}^{1})+\mathcal{L}(C_{t}^{2})\\
 & \ge2\limsup_{t\to0}f(t)\mathcal{L}(C_{0}).
\end{align*}
which is a contradiction as $\mu_K(K)=\mu_K(C_0)=1$ and $\mu_K$ and $\mathcal{L}\big|_K$ are mutually absolutely continuous.
\end{proof}

Using Lemma \ref{lem:strong non-branching} one can extend the proposition to general achronal target measures.

\begin{prop}
The previous preposition also holds for probability measures $\nu$ supported in an achronal set. Furthermore,
for the unique dynamical optimal coupling $\Pi$ and any $c_L$-cyclically monotone set $C\subset\operatorname{supp}\pi$
with $\pi(C)=1$ where $\pi = (\ev_0,\ev_1)_\sharp\Pi$ it holds 
\[
\mathcal{L}(C_t) \ge f(t) \mathcal{L}(C_0).
\]
\end{prop}
\begin{proof}

Assume $\pi$ is a $c_L$-optimal coupling for $(\mu,\nu)$ and choose a $c_L$-cyclically
monotone measurable set $C\subset\operatorname{supp}\pi$ of full $\pi$-measure. 

Set $C^0 = \Delta \cap C$ and $C^{>0} = C \backslash \Delta$ where $\Delta$ is the diagonal
in $M \times M$. The intersection $C^0_0 \cap C^{>0}_0$ is $\mu$-negligible. 

Indeed, all points in the intersection $C^0_0 \cap C^{>0}_0$ would have a minimizer passing through that point
which intersects two (necessarily distincts) points in $C^0_0$ and $C^{>0}_1$. Hence the 
set must be $\mathcal{L}$-negligible which also shows $\mu(C^0_0 \cap C^{>0}_0)=0$.

Observe that the claim implies that $\pi$ is induced by a transport map if and only if $\pi$ restricted to 
$C^{>0}$ is induced by a transport map. If either of the cases holds then $\pi$ must be unique.

So without loss of generality one can assume that $\pi$ is concentrated away from the diagonal $\Delta$.
In this case $\pi$ must be concentrated on $\bigcup_{\tau_{0}\in\mathbb{Q},n\in\mathbb{N}} \Omega_{\tau_{0},n}$ where
\[
\Omega_{\tau_{0},n}=\left\{\tau\le\tau_0-\frac{1}{n}\right\}\times\left\{\tau\ge\tau_0+\frac{1}{n}\right\}.
\]

Furthermore, $\pi$ is induced by a transport map if and only if for each $\tau_0 \in \mathbb{Q}$ and $n \in \mathbb{N}$,
$\pi|_{\Omega_{\tau_0,n}}$ is either 
the zero measure or induced by a transport map. Thus one may assume that $\pi$ is supported in $\Omega_{\tau_0,n}$ for some 
$\tau_0 \in \mathbb{Q}$ and $n \in \mathbb{N}$. 

Let $\sigma:\Gamma \to (0,1)$ be measurable with $\tau(\gamma_\sigma)=\tau_0$ whenever 
$\gamma(0)\le\tau_0\le\gamma(1)$. Then given an optimal dynamical coupling $\Pi$ one obtains an intermediate measure
$\mu_\sigma$ which is supported in the time-slice $\{\tau=\tau_0\}$. By Proposition \ref{P20} for any
$\Pi$ there is a unique optimal coupling $\pi_\sigma$
between $\mu$ and $\mu_\sigma$ and a measurable map $T_\sigma$ such that 
$\pi_\sigma = (\operatorname{id}\otimes T_\sigma)_\sharp \mu$. 

We claim that $\Pi$ is unique among the dynamical couplings representing $\pi$. Assume $\mu_\sigma^{'}$, $\pi_\sigma^{'}$ 
and $T_\sigma^{'}$ are obtained from a distinct optimal dynamical coupling $\Pi'$. In this case the maps $T_\sigma$ 
and $T_\sigma^{'}$ do not agree on a set of positive $\mu$-measure. By construction the measure 
$\frac{1}{2}(\pi_\sigma+\pi_\sigma^{'})$ is the unique optimal coupling between $\mu$ and 
$\frac{1}{2}(\mu_\sigma+\mu_\sigma^{'})$ which is induced by a transport map. However, this is only possible 
if $T_\sigma$ and $T_\sigma^{'}$ agree $\mu$-almost everywhere. This is a contradiction and shows
that the dynamical coupling $\Pi$ representing $\pi$ is unique. 

Note that for $\pi$-almost all $(x,y)\in M\times M$ the point $T_\sigma(x)$ is on a geodesic connecting $x$ and $y$. 
Since the value of $T_\sigma$ is unique almost everywhere and geodesics are non-branching, for $\mu$-almost all $x\in M$ 
there can be at most one geodesics $\gamma$  with $\gamma_0=x$ and $\gamma_\sigma = T_\sigma(x)$. In particular, 
for $\mu$-almost all $x\in M$ there is a unique $(x,y)\in \operatorname{supp}\pi$. But then $\pi$ is induced by a 
transport map and hence the unique optimal coupling between $\mu$ and $\nu$.

It remains to show that the interpolation inequality holds as well: Let $\Pi$ be the unique dynamical optimal coupling 
and $\pi$ be the unique induced optimal coupling. 

Let $\chi: \Gamma \to (1-\epsilon,1]$ be a measurable map such that for a set $\Gamma'$
of full $\Pi$ measure the set $\tau\circ\chi(\Gamma')$ is countable and whenever $\gamma(1)=\eta(1)$ then 
$\tau(\gamma_\chi)=\tau(\eta_\chi)$.

Let $\mu_\chi$ be the intermediate measures obtained from $\chi$. Then $\mu_\chi$ is concentrated in countably many
time-slices $\{\tau=\tau_k\}_{k\in\mathbb{N}}$. Observe that the interpolation property holds when we 
restrict the coupling to $M \times \{\tau_k\}$. Since the endpoints for two different time-slices
are disjoint Lemma \ref{lem:strong non-branching} implies that the interpolated points never intersect. 
Thus if $C$ is a $c_L$-cyclically monotone subset of $\operatorname{supp}\pi$ of full $\pi$-measure 
then the set
\[
C^{\chi}=\{(\gamma_{0},\gamma_{\chi})\,|\,\gamma\in(\ev\nolimits_{0},\ev\nolimits_{1})^{-1}(C)\}
\]
is $c_L$-cyclically monotone and has full $(\ev_0,\ev_\chi)_\sharp \Pi$-measure and it holds 
\[
\mathcal{L}(C_{t}^{\chi})\ge f(t)\mathcal{L}(C_{0}).
\]

Via approximation it suffices to show the interpolation property assuming $C$ is compact. Observe now that for compact
$C$ and all $\delta>0$ it holds
\[
C^\chi_t \subset (C_t)_\delta
\]
for $\epsilon>0$ sufficiently small. Thus

\begin{align*}
\mathcal{L}(C_{t}) & =\lim_{\delta\to0}\mathcal{L}((C_{t})_{\delta})\\
 & \ge\limsup_{\epsilon\to0}\mathcal{L}(C_{t}^{\chi})\ge f(t)\mathcal{L}(C_{0}).
\end{align*}\end{proof}

\begin{proof}[Proof of Theorem \ref{thm4}]

The only thing that is left is to show that the intermediate measures $\mu_t = (\ev_t)_\sharp \Pi$ are absolutely
continuous. For this let $C = \supp \pi$ and assume $\mu_t$ was not absolutely continuous. Then there
 is a compact set $\tilde{C}\subset C$ such that  $\mu(\tilde{C}_0)=\mu_t(\tilde{C}_t)>0$ and 
$\mathcal{L}(\tilde{C}_t)=0$. In particular, $\mathcal{L}(\tilde{C}_0)>0$. However, the interpolation property shows 
$ 0 = \mathcal{L}(\tilde{C}_t)\ge f(t)\mathcal{L}(\tilde{C}_0)$ which is clearly a contradiction
and thus proving that $\mu_t$ is absolutely continuous.
\end{proof}

The following corollary turns out to be useful in the next section.

\begin{cor}[Self-Intersection Lemma]\label{cor:self-intersection}
If $\mu$ and $\nu$ are causally related, $\mu$ is absolutely continuous and $\nu$ is supported on an 
achronal set then for all sets $A$ of full $\mu$-measure there is a $t_0 \ll 1$ such that 
the intermediate measures $\mu_t$, $t\in (0,t_0)$ satisfy $\mu_t(A)>0$. In particular, $\mu$ 
and $\mu_t$ cannot be mutually singular.
\end{cor}
\begin{proof}
By restricting $\mu$ we may assume $\mu$ has density by $M$. Then uniqueness of $\mu_t$ implies 
that the density of $\mu_t$ is bounded by $M\cdot f(t)^{-1}$, see \cite[5.15]{kell17}. 
Since $f(t) \to 1$ as $t \to 0$ we see that the densities of $\mu_t$ for sufficiently small 
$t$ can be uniformly bounded. Now the claim follows directly from the  Self-Intersection Lemma
in \cite[Lemma 6.4]{kell17}.
\end{proof}
\begin{remark}
The argument shows that for $\mu=g\mathcal{L}$ and $\mu_t = g_t\mathcal{L}$
one has the estimate 
$$g_t(\gamma_t) \le \frac{1}{f(t)} g(\gamma_0)$$ 
for $\Pi$-almost all $\gamma \in \Gamma$ where $\Pi$ is the unique optimal dynamical coupling
between $\mu$ and $\nu$.
\end{remark}

\subsection{Proof of Theorem \ref{thm5}}

The goal is to reduce the problem to the $1$-di\-men\-sion\-al case and then construct a map from that solution. The proof
is very similar to the proof of Bianchini-Cavalletti \cite{BiaCav} for general non-branching geodesic spaces. However, the
lack of a natural parametrization of lightlike geodesics prevents a direct application of their proof. One of the features of the proof will be to show how the 
time function $\tau$ and time-affinely parametrized geodesics can be used to overcome this obstacle and give a complete 
solution to the Monge problem in the relativistic setting. 

Note that the proof shows that the optimal coupling is in general non-unique without assuming 
some relative form of achronality. Indeed, in order to prove uniqueness using the reduction to 
the $1$-dimensional setting on a set of full measure there must be an almost everywhere defined injective 
map from the set of transport rays to $M$ which corresponds to the target of the transport.

By \cite[Proposition 2.7]{suhr16} one knows that the any optimal coupling is concentrated on a measurable $c_L$-cyclically 
monotone set $C$.

\begin{definition}
[Maximal $c_L$-cyclically monotone set] A set $A\subset M\times M$ is {\it maximal $c_{L}$-cyclically
monotone in a set $\Sigma\subset\{c_{L}\le0\}$}
if it is $c_L$-cyclically monotone and is maximal with respect to inclusion among subsets
of $\Sigma$.  
\end{definition}

It is not difficult to see that a maximal $c_{L}$-cyclically monotone set $A$ must be closed if $\Sigma$ is closed. One calls any
maximal element $A_{max}$ of a $c_{L}$-cyclically monotone set $A$ a {\it maximal hull}. Note that the maximal hull is in general not unique. 

\begin{lemma}
Every $c_{L}$-cyclically monotone set $A\subset\{c_{L}\le0\}$ is
contained in a maximal $c_{L}$-cyclically monotone set $A_{max}\subset\{c_{L}\le0\}$.
In particular, if $(\mu,\nu)\in\mathcal{P}^+_\tau(M)$, then any
optimal coupling is supported in a maximal $c_{L}$-cyclically monotone
set $A_{max}\subset\{c_{L}\le0\}$.
\end{lemma}
\begin{proof}
Just observe that if $\{A_{i}\}_{i\in I}$ is a chain of $c_{L}$-cyclically
monotone sets in the closed set $\{c_{L}\le0\}$ then 
$$\tilde{A}=\bigcup_{i\in I}A_{i}$$
is maximal in $\{c_L\le0\}$ and $c_{L}$-cyclically monotone. Thus Zorn's Lemma gives
the existence of a maximal element $A_{max}$ with 
$$A\subset A_{max}\subset\{c_{L}\le0\}.$$
The last statement follows by observing that a coupling with finite
cost must have support in $\{c_{L}\le0\}$.
\end{proof}

Let $A_{max}$ be a maximal $c_L$-cyclically monotone hull of the support of an 
optimal coupling $\pi$ of $\mu$ and $\nu$ in $J^+$. Further let $\Pi$ be a dynamical optimal coupling of $(\mu,\nu)$. 

\begin{lemma}\label{lem:A_x-geodesics}
For any point $(x,y)\in A_{max}$ and any point $z\in \gamma\in \Gamma_{x\rightarrow y}$ one has $(x,z),(z,y)\in A_{max}$. 
\end{lemma}

\begin{proof}
Let $\{(x_i,y_i)\}_{1\le i\le N}\subset A_{max}$. Then one has 
\begin{align*}
c_L(x,z)&+c_L(z,y)+\sum_{i=1}^N c_L(x_i,y_i) =c_L(x,y)+\sum_{i=1}^N c_L(x_i,y_i)\\
&\le c_L(x,y_1)+\sum_{i=1}^{n-1} c_L(x_i,y_{i+1}) +c_L(x_N,y)\\
&\le c_L(x,y_1)+\sum_{i=1}^{n-1} c_L(x_i,y_{i+1}) +c_L(x_N,z)+c_L(z,y)
\end{align*}
where the next to last inequality follows from the cyclic monotonicity and the last inequality is the triangle inequality for $c_L$. This implies that 
$(x,z)\in A_{max}$. The other case is analogous. Note that $c_L(x,y)<\infty$ since $(x,y)\in J^+$ and thus $c_L(x,z),c_L(z,y)<\infty$.
\end{proof}

Consider the relation $R\subset M\times M$ with 
$$(x,y)\in R:\Leftrightarrow (x,y)\in A_{max} \text{ or }(y,x)\in A_{max}.$$
Set $R_{>1}:=\{(x,y)\in R|\; \exists z\neq x: (x,z)\in R\}$. Then one can assume without loss of generality that $R_{>1}$ has full measure relative to any 
optimal coupling. This follows from the observation that on $R\setminus R_{>1}$ all optimal transports are constant. It is assumed from here on that 
$R=R_{>1}$.

Next define the following two sets:
$$A^+:=\{x\in M|\; \exists z\neq w\in M: (x,z),(x,w)\in A_{max}\text{ and }(z,w)\notin R\}$$
and 
$$A^-:=\{y\in M|\; \exists x\neq z\in M: (x,y),(z,y)\in A_{max}\text{ and }(x,z)\notin R\}.$$

Assume the disintegration of $\pi$ with respect to the first projection
is given by 
\[
\pi=\mu\otimes\pi_{x}.
\]

\begin{lemma}\label{lem:exclude-Aplus}
For $\mu$-almost all $x\in A^{+}$ the measures $\pi_{x}$ are supported
in $\{(x,x)\}$.
\end{lemma}

\begin{proof}
If $A^+$ is $\mu$-negligible there is nothing to prove. Therefore one can assume by \cite[Corollary 3.12]{suhr16} that $\mu(A^+)=1$. After possibly further 
restricting the transport problem one can suppose that 
\[
\operatorname{supp}\pi\subset\{\tau\le\tau_{0}-\varepsilon\}\times\{\tau\ge\tau_{0}+\varepsilon\}
\]
for a sufficiently small $\e>0$. Let now 
$$\Gamma_{\tau_0}:=\{\gamma\in \Gamma|\; \tau(\gamma(0))\le \tau_0\le \tau(\gamma(1))\}$$
and $\sigma:\Gamma_{\tau_0}\to[0,1]$ be the map defined by $\tau(\ev(\gamma,\sigma(\gamma)):=\tau_{0}$. 
Note that $\Pi$ is supported in $\Gamma_{\tau_0}$ under the above assumptions. Then 
$$\mu_{\sigma}:=(\ev\circ(\id,\sigma))_\sharp\Pi$$
is an intermediate measure of $\mu$ and $\nu$ which is supported
in the time-slice $\{\tau=\tau_{0}\}$. Thus by Proposition \ref{P20} there is a unique coupling
which, in addition, is induced by a transport map $T_{\sigma}$. The
assumption shows that for $\mu$-almost all $x\in A^{+}$ the (unique)
geodesics connecting $x$ and $T_{\sigma}(x)$ never intersects $A^{+}$.
Thus $\mu_{\sigma,t}(A^{+})=0$ for the any intermediate measure $\mu_{\sigma,t}$
between $\mu$ and $\mu_{\sigma}$. This, however, violates Corollary
\ref{cor:self-intersection}.
\end{proof}
\begin{remark}
Under the assumptions of Theorem \ref{thm5} it is even possible to show $\mathcal{L}(A^{+})=0$. For 
this one needs to know whether any coupling $\tilde{\pi}$
concentrated in $A_{\max}\cap(A^{+}\times\{\tau=\tau_{0}\})$ is
induced by a transport map. By Theorem \ref{thm4} a sufficient condition
would be that $\tilde{\pi}$ is optimal. 
\end{remark}

If $\pi_{x}=\delta_{x}\otimes\delta_{x}$ for all $x$ in a measurable
set $A\subset M$ then $\pi\big|_{A\times M}=(\operatorname{id}\times\operatorname{id})_{\sharp}\mu\big|_{A}$.
Thus in the following one will always assume that  $\pi_{x}\ne\delta_{x}\otimes\delta_{x}$
for $\mu$-almost all $x\in M$. In particular, the measures
$\mu$ and $(p_{1})_{\sharp}(\pi\big|_{M\times M\backslash\Delta})$
are mutually absolutely continuous. In combination with Lemma \ref{lem:exclude-Aplus} one concludes that $\mu(A^{+})=0$.

Given a symmetric relation $R\subset M\times M$ let the {\it domain of $R$} be defined by
$$\dom(R):=p_{1}(R).$$
In the following one will use the following short hand notation to
define a new symmetric relations $R'\subset R$: $R'$ {\it is given by $\dom(R')=A$} for a subset $A\subset\dom R$ if 
\[
R'=A\times A\cap R.
\]
One easily verifies that $\dom R'$ is indeed equal to $A$. Furthermore,
if $A$ is (Borel) measurable then $R'$ is (Borel) measurable or analytic
if $R$ is (Borel) measurable or analytic, respectively. 

Let $R_{red}$ be obtained by requiring $\dom(R_{red})=\dom(R)\setminus A^+\cup A^-$. Then one may verify that $R_{red}$ is an equivalence relation.

Decompose $\mu$ into two measures $\mu_1$ and $\mu_2$ such that $\mu_1$ is concentrated $\dom(R_{red})$ and $\mu_2$ on $A^-$. Choose an 
optimal coupling $\pi$ along the same decomposition $\mu = \mu_1 + \mu_2$. Denote the second marginals by $\nu_1$ and $\nu_2$, respectively. By 
the definition of $A^-$ one sees that $\pi_2$ is concentrated on the diagonal. Thus if one finds an optimal coupling $\tilde{\pi}_1$ between $\mu_1$ 
and $\nu_1$ which is induced by a transport map then $\tilde\pi = \tilde\pi_1 + \pi_2$ is an optimal coupling between $\mu_1$ and $\nu_1$ which is 
induced by transport maps. Thus one may assume $\mu(A^\pm)=0$.

\begin{lemma}
There exists a measurable projection $T\colon \text{dom}(R_{red})\to \text{dom}(R_{red})$ with $(x,T(x))\in R_{red}$. 
\end{lemma}

\begin{proof}
Choose an enumeration $\{q_n\}_{n\in \N}$ of $\Q$. Define inductively disjoint relations $\{R_n\}_{n\in\N}$ as follows: Set  $\tilde{R}_0:=R_{red}$. Assume that $R_k$ for $k\le n$ has been
constructed. Define $R_{n+1}$ by 
$$(x,z)\in R_{n+1}:\Leftrightarrow (x,z)\in \tilde{R}_n\wedge\exists y,y'\in \{\tau=q_{n+1}\}: (x,y),(z,y')\in \tilde{R}_n.$$
$R_{n+1}$ is an equivalence relation since it is the intersection of two equivalence relations. 
Thus $\tilde{R}_{n+1}:=\tilde{R}_n\setminus R_{n+1}$ is an equivalence relation. Continuing one obtains a measurable partition $\{R_n\}_{n\in \N}$ of $R_{red}$. 
This follows from the initial assumption that all minimizer are non-constant. 

For all $r\in \R$ there exists a measurable selection 
$$S_r\colon p_1(R_{red}\cap (M\times \{\tau=r\}))\to \dom(R_{red})\cap \{\tau=r\}.$$ 
Define the map 
$$T\colon \dom R_{red}\to \dom R_{red},\; x\mapsto S_{q_n}(x)\text{ for }x\in R_n.$$
\end{proof}

Disintegrate $\mu$ along $T$, i.e. for $\mu_{red}:=(T)_\sharp \mu$ let $\{t_x\}_{x\in \dom(R_{red})}$ be the almost everywhere defined family of 
probability measures on $T(\dom(R_{red}))$ such that 
$$\mu=\mu_{red}\otimes t_x.$$

\begin{lemma}
For $\mu_{red}$-almost all $x\in M$ is the measure $\bar{t}_{x}=\tau_{\#}(t_{x})$
is non-atomic.
\end{lemma}
\begin{proof}
By the assumptions the statement holds for $\pi$ if it holds for
$\pi$ restricted to $M\times M\backslash\Delta$. In particular,
one can assume that for $\pi$-almost all $(x,y)\in M\times M$ one has $\tau(x)<\tau(y)$. 

Assume now for a set $A$ of positive $\mu_{red}$-measure the measure
$\bar{t}_{x}$ has atoms for all $x \in A$. Then there is a compact set $K\subset M$ of positive
$\mu$-measure such that the map 
$$K\to\mathcal{P}(\mathbb{R}),\; x\mapsto\bar{t}_{x}$$ 
is weakly continuous. Thus the function
$$F\colon K\times\mathbb{R}\to[0,1],\; F(x,r)=\bar{t}_{x}(\{r\})$$ 
is upper semi-continuous. In particular, the set 
\[
C=F^{-1}((0,1])
\]
is a Borel set and for each $(x,r)\in C$ the point $r$ is an atom
of $\bar{t}_{x}$. Applying the Selection Theorem \cite[Section 423]{Fremlin2006} to $C$ yields a measurable selection $T\colon p_1(C)\to C$ such that $(x,T(x))\in C$ for all $x\in p_1(C)$.
In particular, $\mu_{red}\big|_K\otimes\delta_{T(x)}$ is non-trivial. Since $t_x$ is atomic for all $x\in K$ one also has
\[
\mu_{red}\big|_K\otimes\delta_{T(x)}\ll\mu_{red}\otimes\bar{t}_{x}.
\]

Translating back to the coupling $\pi$ one sees that there exists a measurable
map $S\colon M\to M$ and a set $\tilde{K}$ of positive $\mu$-measure
such that 
$$\pi':=\frac{1}{\mu(K)}(\operatorname{id}\otimes S)_{\sharp}\mu\big|_{K}\ll\pi$$
and for all $x\ne y\in\tilde{K}$ one has $(x,T(y))\notin R_{red}$
and $\tau(x)<\tau(T(x))$. 

This implies that for any $\sigma:\Gamma\to(0,1)$ the intermediate
measures $\mu_{\sigma}^{'}$ of $(p_{1})_{\sharp}\pi'$ and $(p_{2})_{\sharp}\pi'$
would be mutually singular with respect to $(p_{1})_{\sharp}\pi'$. However, as
in the proof of Lemma \ref{lem:exclude-Aplus}, this yields a contradiction.
\end{proof}
\begin{remark}
A more involved proof shows that $\bar{t}_{x}$ is absolutely
continuous. As this strengthened statement is not needed the details are left to 
the interested reader.
\end{remark}

Disintegrating an optimal coupling $\pi$ along $T\circ p_1$ yields a family of probability measures $\{s_x\}$ such that 
$$\pi=\mu_{red}\otimes s_x,$$ 
where $s_x$ is a probability measure on 
$$(R_x\cap\dom(R_{red}))\times R_x\text{ with }R_x:=p_2((\{x\}\times M)\cap R).$$ 

\begin{lemma} For all $x\in T(\dom(R_{red}))$ the set $R_x$ is diffeomorphic to an interval and the 
time function $\tau$ is injective on $R_x$.
\end{lemma} 
\begin{proof}
From Lemma \ref{lem:A_x-geodesics} and $x\in R_x$ one sees that $R_x$ is formed by the image of geodesics which contain $x$ and meet at most at their endpoints. As $x\in R_x$ is not in $A^+$ or $A^-$, it must be in the interior of $R_x$.
Thus, because geodesics are non-branching and the time function $\tau$ is strictly increasing along causal curves one sees that $R_x$
is the image of precisely one geodesic. 
\end{proof}

Define for $x\in T(\dom(R_{red}))$ the measures $r_x:=(p_2)_\sharp(s_x)$, i.e. 
$$\nu=\mu_{red}\otimes r_x.$$

Next one constructs a transport map for the optimal couplings between $t_x$ and $r_x$. 
By the previous Lemma the measures $t_x$ and $r_x$ are concentrated on a single geodesic 
such that the time function $\tau$ give a uniquely defined parametrization. Thus it suffices 
to solve the one-dimensional optimal transport problem between $t_x$ and $r_x$. First observe the following.

\begin{lemma}\label{L21}
Let $\gamma \in \Gamma$ and $\mu$, $\nu$ be causally related probability measures on $\gamma$. Then any causal coupling $\pi_\gamma$ is optimal. 
\end{lemma}

\begin{proof}
Choose a monotone reparameterization $[0,1]\to[0,1]$ of $\gamma$ to an affine parameter. The cost function for $s,t\in [0,1]$ then is 
$$c_L(s,t)=\begin{cases} c(t-s),\text{ if }s\le t,\\ \infty\text{ else}\end{cases}$$
for some constant $c=c(\gamma)\le 0$. It is now easy to see that any causal coupling is cyclically monotone, i.e. optimal.
\end{proof}

By Lusin's Theorem one can assume that $x\mapsto (t_x,r_x)$ is continuous. Set 
$$m(x,a):=t_x(\tau^{-1}(-\infty,a])$$ 
and 
$$n(x,b):=r_x(\tau^{-1}(-\infty,b]).$$ 
With this define $\phi(x,a)=b$ if $b=\text{argmin}\{m(x,a)\le n(x,b)\}$. 
Observe that $\phi$ is measurable and $(T,\tau)$ is injective on $\dom(R_{red})$ 
so that there is a measurable map $\psi:\dom(R_{red})\to M$  such that
$\psi(y) = \phi(T(y),\tau(y))$ for $\mu_{red}$-almost all $y\in \dom(R_{red})$. 

Again by Lusin's Theorem one may assume $\psi$ is continuous. Define a set 
$\mathcal{T}\subset M\times M$ as follows
$$
\mathcal{T} = \{ (y,z) ~|~ (y,z)\in R, \psi(y)=\tau(z) \}.
$$
Note that $\mathcal{T}$ is analytic and for each $x\in \dom(R_{red})$ there
is exactly one $(x,y)\in \mathcal{T}$. Thus $\mathcal{T}$ agrees on $\dom(R_{red})\times M$
with the graph of a measurable function $\Psi : \dom(R_{red}) \to M$. 

The choice of $\phi$ implies $\Psi_\sharp t_x = r_x$. Thus
$(\operatorname{id}\times\Psi)_\sharp \mu $ is a coupling of $\mu$ and $\nu$.
Since $\Psi$ transports monotonously along each $R_x$
one sees that $\Psi$ is an optimal transport map between $t_x$ and $r_x$. As the 
initial coupling was optimal, we see that along each transport ray the cost is not change
In particular, the coupling $(\operatorname{id}\times\Psi)_\sharp \mu $ is optimal between $\mu$ and $\nu$. 
This finishes the proof of Theorem \ref{thm5}.

\end{document}